\documentclass{amsart} 
\usepackage[utf8]{inputenc} 
\usepackage{lmodern}
\usepackage[a4paper,margin=3cm,headheight=.4in,headsep=12pt,heightrounded]{geometry}
\usepackage{amsmath,amsfonts,amssymb,amscd,mathtools}
\usepackage{graphicx} % \includegraphics 
\usepackage{xcolor}
\usepackage{parskip} % starts paragraphs without spaces
\usepackage{microtype} % typesetting looks better when printed
\usepackage{booktabs} % better looking tables
\usepackage{array} % better arrays (eg matrices) 
\usepackage{paralist} % flexible & customisable lists (eg. enumerate/itemize, etc.)
\usepackage{verbatim} % commenting out blocks of text & for better verbatim
\usepackage{tikz}
\usepackage{tikz-cd}
\usepackage{amsthm}
\usepackage{bbm} % extra mathbold fonts 
\usepackage{wrapfig} % helps have text run around the side of pictures 
\usepackage{setspace}
\usepackage{accents}
\usepackage{fancyhdr}
\usepackage{mathrsfs}
\usepackage{mathdots}
\usepackage{mathtools}
\usepackage{empheq}
\usepackage{hyperref}
\usepackage{cleveref}
\usepackage[style=numeric,backend=biber,maxnames=50]{biblatex}
\usepackage{animate}
\usepackage{placeins}
\usepackage{afterpage}
\usepackage[new]{old-arrows}
\usepackage[font=tiny,labelfont=bf]{caption}

\usetikzlibrary{angles,calc,intersections,quotes,arrows.meta}
\usetikzlibrary {positioning}
\usetikzlibrary{decorations.markings}
\usetikzlibrary{%
  matrix,%
  calc,%
  arrows%
}
\newcommand{\spn}{\text{span}} % span of vectors
\newcommand{\abs}[1]{\left| #1 \right|} % absolute values
\newcommand{\norm}[1]{\left \| #1 \right \|} % norms (two absolute value bars)
\newcommand{\nnorm}[1]{{\left\vert\kern-0.25ex\left\vert\kern-0.25ex\left\vert #1 \right\vert\kern-0.25ex\right\vert\kern-0.25ex\right\vert}}
\newcommand{\comments}[1]{} % a command so I can comment out big blocks of stuff

\newcommand\restr[2]{{
  \left.\kern-\nulldelimiterspace 
  #1 
  \vphantom{\big|} 
  \right|_{#2} 
  }}

\DeclareMathOperator{\trace}{tr}
\DeclareMathOperator{\rk}{rank}
\newtheorem{thm}{Theorem}[section]
\newtheorem{lem}[thm]{Lemma}
\newtheorem{cor}[thm]{Corollary}
\newtheorem{prop}[thm]{Proposition}
\theoremstyle{definition}
\newtheorem{defn}[thm]{Definition}
\newtheorem{exa}[thm]{Example}
\theoremstyle{remark}
\newtheorem{rem}[thm]{Remark}

\DeclareMathOperator{\image}{im}
\DeclareMathOperator{\interior}{int}
\DeclareMathOperator{\End}{End}
\DeclareMathOperator{\Spec}{Spec}

\newcommand{\cL}{\mathcal{L}}
\DeclareMathOperator{\Hom}{Hom}

\newcommand{\K}{\mathbbm{k}}

\DeclareMathOperator{\M}{M}
\DeclareMathOperator{\GL}{GL}
\DeclareMathOperator{\I}{I}
\DeclareMathOperator{\V}{V}

\newcommand{\Langle}{\mathop{<}\!}
\newcommand{\Rangle}{\!\mathop{>}}
\newcommand{\mon}{\Langle\underline{x},\underline{y}\Rangle}
\newcommand{\mo}{\Langle\underline{x}\Rangle}

\newcommand{\defeq}{\vcentcolon=}

\crefformat{footnote}{#2\footnotemark[#1]#3}
\hypersetup{
    colorlinks=true,
    linkcolor=blue,
    filecolor=magenta,      
    urlcolor=cyan,
    citecolor=blue,
}
\title{Positivity of hereditary quadratic nc polynomials and \(x^*x\)-convex domains}

\author[A. Islam]{Abdellah Islam}
\address{A. Islam, Department of Mathematics, Drexel University, Philadelphia, PA 19104, USA}
\email{ahi25@drexel.edu}

\author[J. Volčič]{Jurij Volčič}
\address{J. Volčič, Department of Mathematics, University of Auckland, Auckland, New Zealand}
\email{jurij.volcic@auckland.ac.nz}

\author[H.J. Woerdeman]{Hugo J. Woerdeman}
\address{H.J. Woerdeman, Department of Mathematics, Drexel University, Philadelphia, PA 19104, USA}
\email{hugo@math.drexel.edu}

\thanks{HW is partially supported by NSF grant DMS-2348720.}



\begin{document}

\begin{abstract} 

    Noncommutative polynomials, elements of the free algebra generated by several noncommuting variables \(x_1,...,x_g,x_1^*,...,x_g^*\), are of importance in noncommutative algebra, free analysis and real algebraic geometry. 
    We allow matrix coefficients and view these noncommutative polynomials as functions in several size-free matrix variables. 
    The \emph{free locus} of a noncommutative polynomial \(f\) is the sequence of hypersurfaces \(\mathscr{Z}_n(f)=\{X\in\M_n(\mathbb{C})^g:\det f(X) = 0\}\).
    When \(f\) is Hermitian, the \emph{positivity domain} of \(f\) is the sequence of semialgebraic sets \(\mathcal{D}_n(f)=\{X\in\M_n(\mathbb{C})^g:f(X)\text{ is positive semidefinite}\}\). 
    
    We show for well-behaved cases that the Euclidean boundary of the positivity domain is Zariski dense in the free locus.
    This leads to a certificate for equality of positivity domains of coefficient-irreducible unsignatured atomic quadratics.
    Finally, we describe the geometric structure of positivity domains of hereditary quadratic polynomials through \(\Gamma\)-convexity.
\end{abstract}

\maketitle

%%%%%%%%%%%%%%%%%%%%%%%%%%%%%%%%%%%%%%%%%%%%%%%%%%%%
% INTRODUCTION
%%%%%%%%%%%%%%%%%%%%%%%%%%%%%%%%%%%%%%%%%%%%%%%%%%%%
\section{Introduction}

The study of polynomial equalities and inequalities have its roots in classical algebraic geometry and real algebraic geometry, respectively. 
In the direction of classical algebraic geometry, a fundamental result is \emph{Hilbert's Nullstellensatz} \cite{H93} characterising the interplay between commutative algebra and affine geometry.
In the direction of real algebraic geometry, the \emph{real nullstellensatz} \cite[Theorem 4.1.4]{BCR98} bridges commutative algebra to the theory of semialgebraic sets.
This real-algebraic perspective allows inequalities due to working over the real numbers, so one can ask questions of positivity. 
Various \emph{positivstellens\"{a}tze}, algebraic certificates that can detect positivity of polynomials, came out of this theory.
Some first examples are the Krivine-Stengle positivstellensatz \cite{K64, S74}, Schm\"{u}dgen's positivstellensatz \cite{S91} and Putinar's positivstellensatz \cite{P93}.

An avenue for generalisation of the above is the introduction of noncommuting variables, in the sense that \(xy\) and \(yx\) are distinct polynomials. 
One can evaluate these polynomials on matrix tuples and talk of vanishing in this setting;
see for instance Amitsur's noncommutative nullstellensatz \cite{A57} where noncommutative polynomials vanishing at matrix tuples are characterised, i.e. tuples of matrices \((X_1,\ldots,X_g)\) with \(f(X_1,\ldots,X_g)=0\).
This is not the only possible interpretation of vanishing. 
For example on the geometric side of the nullstellensatz, Bergman proved a nullstellensatz for noncommutative polynomials \(f\) evaluated at a matrix tuple vanishing in a direction \(v\), i.e., \((X_1,\ldots,X_g)\) a tuple of matrices and \(v\) a nonzero vector such that \(f(X_1,\ldots,X_g)v=0\) \cite{HM04}.
On the algebraic side of the nullstellensatz, an analogue of radical ideals with the nullstellensatz property was developed in \cite{KVV17} for rationally resolvable ideals.
We will focus on the following interpretation of vanishing: given a noncommutative polynomial \(f\) the \emph{free locus} of \(f\) is the set of all matrix tuples \((X_1,\ldots,X_g)\) for which \(f(X_1,\ldots,X_g)\) is singular. 
That is, we are concerned with the vanishing of the determinant of \(f\). 
Notice that the size of the matrices in question is never specified; this is a feature of the free setting, where our evaluations are size-free.
In this way, one can view free loci as sequences of affine hypersurfaces.
A first treatment of this type of vanishing can be found in \cite{HKV18} where nullstellens\"{a}tze for free loci provide a bridge between the geometric picture of hypersurfaces and factorisation of the noncommutative polynomials in the free algebra.
The theory of free loci can be treated real-algebraically as well, as seen in  \cite{HKV20} where this point of view is utilised to prove a real nullstellensatz; 
the key feature required here is the structure of an antilinear involution on the free locus. 
Not only does this give a better grasp on the geometry, but we also gain an analytic flavour by thinking of the involution as an adjoint operation.

The theory of positivity for noncommutative polynomials is more recent than its commutative counterpart, but rich in application. 
Given a noncommutative polynomial \(f\) with real coefficients, evaluation at a symmetric tuple of matrices will produce a symmetric matrix output.
So we can ask about positivity of the output, leading to the theory of positivity domains, i.e., the set of all symmetric matrix tuples \((X_1,\ldots,X_g)\) for which \(f(X_1,\ldots,X_g)\) is positive semidefinite.
Positivity domains are related to semialgebraic sets in real algebraic geometry, and they rose to relevance due to spectrahedra, the feasibility regions of semidefinite programs:
one can write such a region as the positivity domain of a linear matrix polynomial.
A linear positivstellensatz is given in \cite{HKM12a} and a convex positivstellensatz is given in \cite{HKM12b}.
The convex setting resembles the classical positivstellens\"{a}tze in the sense of utilising a noncommutative analogue of the quadratic module. 
In general, these results are more analytic than the nullstellensatz problems;
for example, the proof of the convex positivstellensatz utilises the GNS construction, a correspondence between \(*\)-representations and positive linear functionals.
A real-algebraic treatment of convex positivity domains can be found in \cite{HKMV20}.
The nullstellens\"{a}tze and positivstellens\"{a}tze find applications in semidefinite optimization, control theory, free probability and quantum information theory.
The theory intersects naturally with free analysis, noncommutative algebra and representation theory.

The main contributions of this paper are to positivity domains of hereditary quadratic polynomials. 
Here, hereditary means that all instances of \(x_j^*\) appear to the left of \(x_i\) terms in each monomial. 
For example, \(x_2^*x_1^*x_3+x_3^*x_1+2\) is hereditary, but \(x_2^*x_1x_3^*\) is not. 
We begin this article by outlining some of the basic theory of noncommutative polynomials, free loci and positivity domains in Section \ref{NotTerm}.
The geometry of the free locus when equipped with a natural involution map is understood through the lens of algebraic geometry.
We include the setup and details in Section \ref{RVarieties}, starting by placing free loci firmly in the theory of varieties and then describing the real points of the free locus. 
This allows us to give a geometric relationship between the positivity domain and the free locus in Theorem \ref{posdense}.
Section \ref{PosDomQuadPoly} contains our main result, Theorem \ref{positivstellensatz}.
This theorem gives a correspondence between similarity of self-adjoint hereditary quadratic matrix polynomials and equality of their positivity domain.
While convex positivity domains have been explored in the literature (as detailed in the previous paragraph), less is known of non-convex domains. 
A generalisation of convexity, called \emph{\(\Gamma\)-convexity}, was formulated in \cite{JKMMP20}. 
This is a natural extension of convexity that arises from positivity domains of noncommutative polynomials; 
we make the observation in Section \ref{GammaConvexity} that hereditary quadratic noncommutative polynomials have \(\Gamma\)-convex geometry for a particular \(\Gamma\).

%%%%%%%%%%%%%%%%%%%%%%%%%%%%%%%%%%%%%%%%%%%%%%%%%%%%
% BASIC NOTATION AND TERMINOLOGY
%%%%%%%%%%%%%%%%%%%%%%%%%%%%%%%%%%%%%%%%%%%%%%%%%%%%
\section{Notation and Terminology}\label{NotTerm}
Let \(\K\) be an algebraically closed field of characteristic \(0\). 
Fix a positive integer \(g\).
Let \(\underline{x}=(x_1,\ldots,x_g)\) be a tuple of freely noncommuting variables, and write \(\mo\) for the semigroup freely generated by \(x_1,\ldots,x_g\). 
Elements of \(\mo\) are called \emph{words in the letters} \(x_1,\ldots,x_g\), or \emph{words}, or \emph{monomials}.
Allow the empty word \(\epsilon\) so that \(\mo\) is a monoid.
Define a \textit{noncommutative polynomial} to be an element of the free (associative, unital) \(\K\)-algebra 
\[
    \K\mo = \bigoplus_{w\in\mo}\K w
\]
where the multiplication is done by concatenation on words.
We omit the empty word from our notation, i.e., if \(\alpha\in\K\) then we write \(\alpha=\alpha\epsilon\in\K\mo\).
For each \(n\in\mathbb{N}\), let \(\M_n(-)\) be the functor from rings to rings, taking a ring \(R\) to its \(n\times n\) matrix ring \(\M_n(R)\), and \(\GL_n(-)\) the functor from rings to groups, taking a ring \(R\) to the unit group \(\M_n(R)^\times\) of \(\M_n(R)\). 
We write \(I_n\) and \(0_n\) for the multiplicative and additive identities of \(\M_n(\K)\) respectively, but when the size is clear from context we write \(I\) and \(0\). 
Let \(\M_{d\times e}(R)\) denote the \(R\)-module of \(R\)-linear maps from \(R^e\) to \(R^d\), i.e. \(\M_{d\times e}(R)\defeq\Hom_R(R^e,R^d)\).
A \textit{matrix polynomial} is an element of the \(\K\)-algebra \(\M_d(\K)\otimes_\K\K\mo\).
We write \(0_{d\times e}\) for the additive identity of \(\M_{d\times e}(R)\), and we write \(0\) when the size is clear from context.
The tensor product comes with a \(\K\)-bilinear map \(\M_d(\K)\times\K\mo\xlongrightarrow{\otimes}\M_d(\K)\otimes_\K\K\mo\), and we write \(\otimes(A,w)=A\otimes w\) for any \(A\in\M_d(\K)\) and \(w\in\K\mo\).
Since \(\M_d(\K)\otimes_\K\K\mo\cong\M_d(\K\mo)\) we also write \(Aw\) for \(A\otimes w\); 
in particular, we omit the empty word so that \(A\epsilon=A\). 
For any \(\K\)-vector space \(V\) with \(\K\)-linear endomorphism \(\K\)-algebra \(\End_\K(V)\), we get \(\K\)-algebra morphisms 
\[
    \text{ev}_{\underline{X}}:\K\mo\to \End_\K(V):x_i\mapsto X_i,\quad\forall \underline{X}=(X_1,\ldots,X_g)\in\left(\End_\K(V)\right)^g.
\]
This induces \(\K\)-algebra morphisms \(\pi_{\underline{X}}:\M_d(\K)\otimes_\K\K\mo\to \M_d(\K)\otimes_\K \End_\K(V)\) for each tuple \(\underline{X}\in\left(\End_\K(V)\right)^g\), which we refer to as evaluation maps:

\[
    \begin{tikzcd}[row sep=huge, column sep=huge]
        \M_d(\K)\times\K\mo \arrow{r}{\otimes} \arrow[swap]{d}{\text{id}\times\text{ev}_{\underline{X}}} & \M_d(\K)\otimes_\K\K\mo \arrow[d,dashed,"\exists! \, \pi_{\underline{X}}"] \\
        \M_d(\K)\times\End_\K(V) \arrow{r}{\otimes} & \M_d(\K)\otimes_\K\End_\K(V)
    \end{tikzcd}
\]

For any \(p\in \M_d(\K)\otimes_\K\K\mo\), we write \(p(\underline{X})\defeq\pi_{\underline{X}}(p)\) for \(\underline{X}\in\left(\End_\K(V)\right)^g\).
Implicit in the notation \(p(\underline{X})\) is a choice of a target \(\K\)-vector space \(V\), and the evaluation maps depend on this space.
So we may evaluate matrix polynomials at any tuples of endomorphisms on any \(\K\)-vector space \(V\).
In particular we will consider \(V=\K^n\) so that \(\End_\K(V)\cong\M_n(\K)\), and \(\M_d(\K)\otimes_\K \End_\K(V)\) can be identified with \(\M_d(\K)\otimes_\K\M_n(\K)\cong\M_{dn}(\K)\) via the Kronecker product (which we will also denote by \(\otimes\)).

There is a natural degree grading for matrix polynomials: the \emph{length} \(\ell(w)\) of a word \(w\in\mo\) is the number of letters in the word, and the \emph{degree} \(\deg(p)\) of a matrix polynomial \(p\in\M_d(\K)\otimes_\K\K\mo\) is given by
\[
    \deg(p) \defeq \max\left\{\ell(w):C_w\neq0,\quad\text{where}\quad p = \sum_{w\in\mo}C_w\otimes w\right\} 
\]
using that \(\M_d(\K)\otimes_\K\K\mo\) has \(\K\)-basis \(\{E_{ij}\otimes w:i,j\in\{1,\ldots,d\},\,w\in\mo\}\) where \(E_{ij}\) is the \(d\times d\) matrix of all zeros, except the \((i,j)\)-th entry is \(1\).

The above construction is also possible with rectangular matrix coefficients; instead of a tensor product of \(\K\)-algebras, consider the \(\K\)-vector space \(\M_{d\times e}(\K)\otimes_\K\K\mo\). 
While we lose ring structure on each \(\M_{d\times e}(\K)\otimes_\K\K\mo\), we multiply elements of \(\bigsqcup_{d,e\in\mathbb{N}}\M_{d\times e}(\K)\) by composition of \(\K\)-vector space morphisms so that \(fg\) is well-defined when \(f\in\M_{d\times e}(\K)\otimes_\K\K\mo\) and \(g\in\M_{e\times e'}(\K)\otimes_\K\K\mo\).

Degree 1 matrix polynomials are called \textit{linear pencils}, or \textit{pencils}. 
That is, matrix polynomials of the form 
\[
    L=A_0+A_1x_1+A_2x_2+\cdots+A_gx_g,\quad\text{where } A_0,A_1,\ldots,A_g\in \M_d(\K).
\]
If \(A_0=I\), we say that \(L\) is a \textit{monic} pencil.
If \(A_0=0\), we say that \(L\) is a \textit{truly linear} pencil.

The following lemma is a well-known result in ring theory. 
We state it here since this result is used often in later arguments.

\begin{lem}\label{diagunit}
    Let \(R\) be a ring and \(a_1,\ldots,a_d\in R\) such that the diagonal matrix \(f=
    \begin{pmatrix}
    a_1 &  &  \\ 
     & \ddots & \\
     & & a_d 
    \end{pmatrix}\) is a unit in \(\M_d(R)\). 
    Then \(a_1,\ldots,a_d\) are units in \(R\).
\end{lem}

\begin{proof}
    Let \(j\in\{1,\ldots,d\}\).
    Since \(f\) is a unit, there exists \(A\in\M_{j-1}(R)\), \(b\in R\) and \(C\in\M_{d-j}(R)\) such that 
    \[
        \begin{pmatrix} 
            A & \bullet & \bullet \\
            \bullet & b & \bullet \\
            \bullet & \bullet & C
        \end{pmatrix}
        \begin{pmatrix}
            a_1 &  &  \\ 
             & \ddots & \\
             & & a_d 
        \end{pmatrix} = I \quad\text{and}\quad
        \begin{pmatrix}
            a_1 &  &  \\ 
             & \ddots & \\
             & & a_d 
        \end{pmatrix}
        \begin{pmatrix} 
            A & \bullet & \bullet \\
            \bullet & b & \bullet \\
            \bullet & \bullet & C
        \end{pmatrix} = I.
    \]
    From the \((j,j)\)-th entry of each of the above two equations, we get \(ba_j=1\) and \(a_jb=1\). 
    So \(a_j\) is a unit.
\end{proof}

With the goal of describing a factorisation theory for matrix polynomials, we require the following definitions of irreducibility, fullness and stable associativity.
While Cohn \cite[Theorem 2.5.1, Corollary 2.5.2]{Cohn06} treats the factorisation theory for matrix rings over any semifir (semi free ideal ring: a ring in which every finitely generated left or right ideal is free of unique rank), we restrict our attention to matrices over \(\K\mo\), which is a fir (free ideal ring: a ring in which every left or right ideal is free of unique rank).

\begin{defn}[{\cite[Section 3.2]{Cohn06}}]\label{atom}
    Let \(R\) be a semifir. 
    We say that \(f\in\M_d(R)\) is an \emph{atom} if \(f\) is nonzero, not a unit in \(R\) and cannot be written as a product of non-units in \(R\).
    We also say that \emph{\(f\) is atomic}.
\end{defn}

\begin{lem}\label{diagatom}
    Let \(R\) be a semifir, \(a_1,\ldots,a_r\in \M_t(R)\) and \(d=rt\). 
    If \(f=\begin{pmatrix}
        a_1 & & \\
        & \ddots & \\
        & & a_r
    \end{pmatrix}\in \M_d(R)\) is an atom, then exactly one of \(a_1,\ldots,a_r\) is an atom and the rest are units.
\end{lem}

\begin{proof}
    If \(a_1,\ldots,a_r\) are all units, then \(f\) is a unit, contradicting that \(f\) is an atom.
    So there is \(j\in\{1,\ldots,r\}\) for which \(a_j\) is not a unit. 
    Let \(f_1= \begin{pmatrix}
        a_1 & & \\
         & \ddots & \\
         & & a_{j-1}
    \end{pmatrix}\) and \(f_2 = \begin{pmatrix}
        a_{j+1} & & \\
        & \ddots & \\
        & & a_r
    \end{pmatrix}\), then 
    \[
        f = \begin{pmatrix}
            f_1 & 0 & 0 \\
            0 & 1 & 0 \\
            0 & 0 & f_2
        \end{pmatrix}\begin{pmatrix}
            I & 0 & 0 \\
            0 & a_j & 0 \\
            0 & 0 & I
        \end{pmatrix}.
    \]
    Since \(f\) is an atom, either \(\begin{pmatrix}
            f_1 & 0 & 0 \\
            0 & 1 & 0 \\
            0 & 0 & f_2
        \end{pmatrix}\) is a unit or \(\begin{pmatrix}
            I & 0 & 0 \\
            0 & a_j & 0 \\
            0 & 0 & I
        \end{pmatrix}\) is a unit. 
        By Lemma \ref{diagunit}, \(\begin{pmatrix}
            I & 0 & 0 \\
            0 & a_j & 0 \\
            0 & 0 & I
        \end{pmatrix}\) is not a unit since \(a_j\) is not a unit. 
        So \(\begin{pmatrix}
            f_1 & 0 & 0 \\
            0 & 1 & 0 \\
            0 & 0 & f_2
        \end{pmatrix}\) must be a unit; 
        again by Lemma \ref{diagunit}, \(a_i\) must be a unit for all \(i\neq j\), i.e., \(a_j\) is the only nonunit.
        Now suppose \(a_j=uv\), then 
        \[
            f = \begin{pmatrix}
                f_1 & & \\
                & u & \\
                & & f_2
            \end{pmatrix}\begin{pmatrix}
                I & & \\
                 & v & \\
                 & & I
            \end{pmatrix}.
        \]
        Since \(f\) is an atom, either \(\begin{pmatrix}
                f_1 & & \\
                & u & \\
                & & f_2
            \end{pmatrix}\) or \(\begin{pmatrix}
                I & & \\
                 & v & \\
                 & & I
            \end{pmatrix}\) is a unit. 
            So by Lemma \ref{diagunit}, either \(u\) or \(v\) is a unit. 
            So \(a_j\) is an atom.
\end{proof}

\begin{defn}[{\cite[Section 0.1]{Cohn06}}]\label{full}
    We say that \(f\in \M_d\left(\K\mo\right)\) is \textit{full} if it cannot be factored as \(f=f_1f_2\) for \(f_1\in \M_{d\times r}\left(\K\mo\right)\) and \(f_2\in \M_{r\times d}\left(\K\mo\right)\) with \(r<d\).
\end{defn}

\begin{rem}\label{atomfull}
    Observe that if \(A\in\M_d(\K\mo)\) were not full then \(A=BC\) where \(B\in\M_{d\times r}(\K\mo)\) and \(C\in\M_{r\times d}(\K\mo)\) with \(r<d\), so that
    \[
        A = \begin{pmatrix} B & 0_{n\times(n-r)}\end{pmatrix} \begin{pmatrix} C\\0_{(n-r)\times n}\end{pmatrix}.
    \]
    Consequently, atoms are full.
\end{rem}

By \cite[Lemma 2.2]{HKV20}, \(f\) is full if and only if there exists \(n\in\mathbb{N}\) and \(\underline{X}\in\M_n(\K)^g\) such that \(\det f(\underline{X})\neq0\).

\begin{defn}[{\cite[Section 0.5]{Cohn06}}]\label{stabass}
    We say that \(f_1\in \M_{d_1}(\K\mo)\) and \(f_2\in \M_{d_2}(\K\mo)\) are \textit{stably associated} if there exist \(P,Q\in\text{GL}_{d_1+d_2}(\K\mo)\) such that \(P(f_1\oplus I_{d_2})=(I_{d_1}\oplus f_2)Q\). 
    That is, the following diagram commutes with exact rows:
    \[
        \begin{tikzcd}[row sep=huge, column sep=huge]
            0 \arrow[r] & \K\mo^{d_1}\oplus\K\mo^{d_2} \arrow[d,swap, "f_1\oplus I_{d_2}"] \arrow[r, "Q"] & \K\mo^{d_1}\oplus\K\mo^{d_2} \arrow[d, "I_{d_1}\oplus f_2"] \arrow[r] & 0 \\
            0 \arrow[r] & \K\mo^{d_1}\oplus\K\mo^{d_2} \arrow[r, "P"] & \K\mo^{d_1}\oplus\K\mo^{d_2} \arrow[r] & 0
        \end{tikzcd}
    \]
\end{defn}
Equivalently by \cite[Theorem 0.5.3]{Cohn06}, \(f_1\) and \(f_2\) are stably associated if and only if
\[
    \K\mo{}^{d_1}/\left(f_1\cdot\K\mo{}^{d_1}\right) \cong \K\mo{}^{d_2}/\left(f_2\cdot\K\mo{}^{d_2}\right),
\]
which follows from universality of the cokernels of \(f_1\oplus I_{d_2}\) and \(I_{d_1}\oplus f_2\).

Every full matrix polynomial \(f\in\M_d(\K\mo)\) splits into atomic factors, and this factorisation is unique up to stable associativity of the factors \cite[Proposition 3.2.9]{Cohn06}.
In particular when \(d=1\), any \(f\neq0\) is full so all nonzero noncommutative polynomials admit an atomic factorisation unique up to stable associativity.

By \cite[Theorem 5.8.3]{Cohn06}, every matrix polynomial \(f\in \M_d(\K)\otimes_\K\K\mo\) is stably associated to a linear pencil \(L=A_0+A_1x_1+\cdots+A_gx_g\) such that 
\[
    \begin{pmatrix} A_{1} & \cdots & A_g\end{pmatrix} \quad\text{and}\quad\begin{pmatrix} A_{1} \\ \vdots \\ A_g\end{pmatrix}
\]
have full rank; we call such a pencil a \emph{linearisation} of \(f\).
A linearisation \(L\) of \(f\) is full if and only if \(f\) is full \cite[Theorem 7.5.13]{Cohn06}, and \(L\) is an atom if and only if \(f\) is an atom \cite[Proposition 3.2.1]{Cohn06}.
Further if \(\det f(0)\neq0\), one can take \(L\) to be monic.

\begin{defn}\label{FL}
    The \textit{free locus} of a matrix polynomial \(f\in \M_d(\K)\otimes_\K\K\mo\) is the set
    \[
        \mathscr{Z}(f)=\bigcup_{n\in\mathbb{N}}\mathscr{Z}_n(f),\quad\text{where}\quad\mathscr{Z}_n(f)=\{\underline{X}\in \M_n(\K)^g:\det f(\underline{X})=0\}.
    \]
    We also say that \(\mathscr{Z}_n(f)\) is the \(n\)-th level of the free locus of \(f\).
\end{defn}

At each \(n\in\mathbb{N}\), the set \(\mathscr{Z}_n(f)\) is an affine variety carved out by a polynomial \(f_n\in k[X_1,\ldots,X_g]\), where we view each \(X_k=(x_{ij})\) as a collection of \(n^2\) commuting variables, and \(f_n\) is the polynomial \(\det f(X_1,\ldots,X_g)\).
Note also that \(\mathscr{Z}(f)\neq\bigcup_{n\in\mathbb{N}}\M_n(\K)^g\) if and only if \(f\) is full \cite[Lemma 2.2]{HKV20}.

\begin{prop}
    If two matrix polynomials \(f\in\M_{d_1}(\K\mo)\) and \(g\in\M_{d_2}(\K\mo)\) are stably associated, then \(\mathscr{Z}(f)=\mathscr{Z}(g)\).
\end{prop}

\begin{proof}
    Suppose \(f,g\) are stably associated, so that there exists \(P,Q\in\GL_{d_1+d_2}(\K\mo)\) such that \(P(f\oplus I_{d_2})=(I_{d_1}\oplus g)Q\).
    Evaluating at \(\underline{X}\in\M_n(\K)^g\),
    \[
        P(\underline{X})\,(f(\underline{X})\oplus I_{d_2n})=(I_{d_1n}\oplus g(\underline{X}))\,Q(\underline{X})
    \]
    and taking the determinant,
    \[
        \det\left(P(\underline{X})\right)\det\left(f(\underline{X})\right) = \det\left( g(\underline{X})\right) \det\left(Q(\underline{X})\right).
    \]
    Since \(P\) is invertible, \(PP'=I\) for some \(P'\in\GL_{d_1+d_2}(\K\mo)\).
    Then evaluating at \(\underline{X}\) and taking determinant, \(\det P(\underline{X}) \det P'(\underline{X}) = 1\).
    Consequently, \(\det P(\underline{X})\) cannot be zero and similarly, \(\det Q(\underline{X})\) cannot be zero.
    It follows that \(\mathscr{Z}(f)=\mathscr{Z}(g)\).
\end{proof}

In particular we see that, for a linearisation \(L\) of a polynomial \(f\), the free locus does not change: \(\mathscr{Z}(L)=\mathscr{Z}(f)\).
A monic pencil is called \emph{irreducible} if its coefficients generate the full matrix algebra, and \emph{reducible} otherwise.
Note that this property was previously called indecomposable in \cite{HKV20}, while in \cite{HKV18} and \cite{HKMV20} it is referred to as irreducible.
Every monic pencil \(L\) is similar to an upper triangular
\[
    \begin{pmatrix}
        L_1 & \star & \cdots & \star \\
        0 & \ddots & \ddots & \vdots \\
        \vdots & \ddots & \ddots & \star \\
        0 & \cdots & 0 & L_r
    \end{pmatrix}
\]
where each \(L_1,\ldots,L_r\) is either 1 or irreducible (mentioned in \cite[(2.1)]{HKV18}; follows from Burnside's theorem \cite[Corollary 5.23]{B14}). 
Then \(\mathscr{Z}(f)=\bigcup_{j\in\{1,\ldots,r\}}\mathscr{Z}(L_j)\) a union of free loci of irreducible pencils whenever \(\det f(0)\neq0\).

%%%%%%%%%%%%%%%%%%%%%%%%%%%%%%%%%%%%%%%%%%%%%%%%%%%%%%%%%%%%%%%%%%%%%%%%%
%%%%%%% THEORY OVER \mathbb{C} NOW INSTEAD OF OVER GENERAL FIELD
%%%%%%%%%%%%%%%%%%%%%%%%%%%%%%%%%%%%%%%%%%%%%%%%%%%%%%%%%%%%%%%%%%%%%%%%%

Set \(\K=\mathbb{C}\). 
Let \(\norm{\cdot}\) denote the operator norm on \(\M_n(\mathbb{C})\) for each \(n\in\mathbb{N}\). 
We also use \(\norm{\cdot}\) to denote the max norm on \(\M_n(\mathbb{C})^g\).
The intended meaning of the norm will always be clear from the context (i.e., from its input). 
Since these norms are closely related, we avoid introducing additional notation to distinguish between them.

Earlier we introduced the notion of irreducibility for a monic pencil \(L\).
A monic pencil \(L\) being reducible is equivalent to finding \(P,Q\in\GL_d(\mathbb{C})\) so that \(PLQ\) is block upper triangular (see, for instance, \cite[Remark 2.6]{HKV20}). 
This holds even when \(L\) is not necessarily monic, but rather \(L(0)\) is invertible.
We say that such a pencil \(L\), with \(L(0)\) invertible, is \emph{irreducible} if the monic pencil \(L(0)^{-1}L\) is irreducible.
Writing \(L=L(0) + A_1x_1+\cdots+A_gx_g\), then \(L\) is reducible if and only if \(L(0)^{-1}A_1,\ldots,L(0)^{-1}A_g\) have a nontrivial common invariant subspace.
In \cite[Definition 2.5]{HKV20}, a more general definition is given, which we will use.

\begin{defn}
    Let \(L\in\M_d(\mathbb{C}\mo)\) be a pencil. 
    We say that \(L\) is \emph{reducible} if there exists \(P,Q\in\GL_d(\mathbb{C})\) for which 
    \[
        PLQ = \begin{pmatrix}
            \star & \star \\
            0 & \star
        \end{pmatrix}
    \]
    where the zero block is of size \(d'\times d''\) with \(d'+d''=d\).
    We say \(L\) is \emph{irreducible} if \(L\) is not reducible.
\end{defn}
In \cite[Definition 2.5]{HKV20}, this property is called \emph{(in)decomposable}.
A pencil \(L=A_0+A_1x_1+\cdots+A_gx_g\) is reducible if and only if there exists \(T\in\GL_d(\mathbb{C})\) for which \(TA_0,TA_1,\ldots,TA_g\) have a nontrivial common invariant subspace, i.e., they generate the full matrix algebra.

Endow \(\mathbb{C}^n\) with standard Hermitian inner product \(\langle\cdot,\cdot\rangle\) and \(\M_n(\mathbb{C})\) with conjugate-transpose involution \(^*\), for each \(n\).
We introduce a reducibility condition for matrix polynomials with respect to this additional structure of involution and orthogonality in the complex setting.
Recall that \(P\in\M_d(\mathbb{C})\) is an orthogonal projection if \(P^2=P\) and \(P^*=P\). 

\begin{defn}\label{irred}
    We say that \(f\in \M_d(\mathbb{C})\otimes_\mathbb{C}\mathbb{C}\mo\) is \emph{coefficient-reducible} if there exists a nonzero proper orthogonal projection \(P\in \M_d(\mathbb{C})\) such that \((I-P)fP=0\).
    If \(f\) is not coefficient-reducible, we say that \(f\) is \emph{coefficient-irreducible}.
\end{defn}

This definition is looser than that of irreducibility; 
by Burnside's theorem, \(f=\sum_{w\in\mo}A_ww\) is coefficient-irreducible if and only if the \(A_w\) generate the full matrix algebra.
So in particular, any irreducible pencil is coefficient-irreducible.
Elements of \(\M_1(\mathbb{C}) \otimes_\mathbb{C} \mathbb{C}\mo\cong\mathbb{C}\mo\) are always coefficient-irreducible: there are no nonzero proper orthogonal projections in \(\M_1(\mathbb{C})\).

\begin{prop}\label{coeff-red rank}
    A matrix polynomial \(f\in\M_d(\mathbb{C})\otimes_\mathbb{C}\mathbb{C}\mo\) is coefficient-reducible if and only if there exist nonzero \(U,V\in\M_d(\mathbb{C})\) with \(\left(\image U\right)^\perp = \image V\) and \(U^*fV=0\).
\end{prop}

\begin{proof}
    Suppose \(U^*fV=0\) for some nonzero \(U,V\in\M_d(\mathbb{C})\) with \(\left(\image U\right)^\perp = \image V\).
    Clearly \(\mathbb{C}^d = \image U \oplus \image V\).
    Let \(P\) be the nonzero proper orthogonal projection onto \(\image V = \left(\image U\right)^\perp\), then \(I-P\) is the projection onto \(\image U = \left(\image V\right)^\perp\).
    For all \(a,b\in\mathbb{C}^d\), writing \(a=u_1+v_1\) and \(b=u_2+v_2\) where \(u_1,u_2\in\image U\) and \(v_1,v_2\in\image V\),
    \[
        b^*(I-P)fPa = u_2^*f(0)+u_2^*fv_1+(0)f(0)+(0)fv_1 = u_2^*fv_1 = 0
    \]
    since \(U^*fV=0\). 
    Consequently, \((I-P)fP=0\).
    
    If \(f\) is coefficient-reducible, there exists a nonzero proper orthogonal projection \(P\) such that \((I-P)fP=0\). 
    Take \(U=I-P\) and \(V=P\).
\end{proof}

\begin{prop}\label{monicirred}
    Let \(L\) be a monic linear pencil. 
    Then \(L\) is coefficient-irreducible if and only if \(L\) is irreducible. 
\end{prop}

\begin{proof}
    Write \(L=I_d+A_1x_1+\cdots+A_gx_g\in\M_d(\mathbb{C}\mo)\), and suppose \(L\) is coefficient-reducible. 
    By Proposition \ref{coeff-red rank} there exist nonzero \(U,V\in\M_d(\mathbb{C})\) with \(\rk U + \rk V = d\) and \(U^*LV = 0\).
    It follows that \(U^*A_jV = 0\) for all \(j\in\{1,\ldots,g\}\), and \(U^*V=0\).
    Let \(S\neq0\) be the image of \(V\), so that \(U^*A_jS=0\) for all \(j\).
    If, for some \(j\), we had \(A_jS\not\subseteq S\), then there exists \(x\in S\) for which \(A_jx\notin S\).
    So \(\ker U^*\) contains \(A_jx\) and also \(S\), since \(U^*V=0\).
    Well \(S\cup\{A_jx\}\) spans a \((\rk V +1)\)-dimensional subspace so \[\rk U^* = d-\dim\ker U^* \leq d-\rk V - 1=\rk U - 1 < \rk U,\]
    contradiction. 
    It follows that \(A_j S\subseteq S\) for all \(j\).
    Consequently, \(S\) is a common invariant subspace for the subalgebra \(\langle A_1,\ldots,A_g\rangle\) in \(\M_d(\mathbb{C})\).
    By Burnside's theorem \cite[Corollary 5.23]{B14}, \(\M_d(\mathbb{C})\) has no nontrivial common invariant subspace. 
    Then the subalgebra \(\langle A_1,\ldots,A_g\rangle\) cannot be equal to \(\M_d(\mathbb{C})\), so \(L\) is reducible.

    Now suppose \(L\) is reducible, so that \(A_1,\ldots,A_g\) do not generate the full matrix algebra.
    By Burnside's theorem, there exists a nontrivial subspace \(S\) such that \(A_jS\subseteq S\) for all \(j\in\{1,\ldots,g\}\). 
    Let \(P\) be the orthogonal projection onto \(S\), then
    \[
        (I-P)LP=(I-P)P+\sum_{j=1}^{g}(I-P)A_jPx_j.
    \]
    Clearly \((I-P)P=0\), and for each \(j\) we have the complex
    \[
        \begin{tikzcd}
            \mathbb{C}^d \arrow[r, "A_jP"] & S \arrow[r, "I-P"] & 0
        \end{tikzcd}.
    \]
    So \((I-P)LP=0\), which means \(L\) is coefficient-reducible.
\end{proof}

Any nonzero scalar multiple of an irreducible (or coefficient-irreducible) linear pencil will be irreducible (or coefficient-irreducible).
So the above proposition is true for all linear pencils whose constant term is a scalar multiple of identity.
If \(f\) is block-upper-triangular then one can use the projection onto the 0 block corner and discover that \(f\) is coefficient-reducible. 
So coefficient-irreducible matrix polynomials are not unitarily similar to a block-upper-triangular matrix over \(\mathbb{C}\mo\).
This leads us to the following factorisation of matrix polynomials. 

\begin{prop}\label{CIfact}
    Let \(f\in\M_d(\mathbb{C})\otimes_\mathbb{C}\mathbb{C}\mo\).
    Then there exists \(f_1,\ldots,f_r\) coefficient-irreducible and a unitary \(U\in\M_d(\mathbb{C})\) so that 
    \[
        U^*fU = \begin{pmatrix}
            f_1 & \star & \cdots & \star \\
            0 & \ddots & \ddots & \vdots \\
            \vdots & \ddots & \ddots & \star \\
            0 & \cdots & 0 & f_r
        \end{pmatrix}.
    \]
\end{prop}

\begin{proof}
    If \(f\) is coefficient-irreducible then we are done. 
    Otherwise, there exists nonzero proper orthogonal projection \(P\) so that \((I-P)fP=0\).
    Let \(S\) be the image of \(P\), then \(S\neq\{0\}\) and \(S\neq\mathbb{C}^d\) since \(P\) is nonzero proper. 
    Let \(U\) be the unitary change of basis to \(S\oplus S^{\perp}\).
    Since \(S^{\perp}=\image(I-P)\) and \((I-P)fP=0\) then 
    \[
        U^*fU=\begin{pmatrix}
            f_1 & \star \\
            0 & f_2
        \end{pmatrix}.
    \]
    Since \(P\) is nonzero proper, each \(f_1\) and \(f_2\) are of strictly smaller size than \(f\).

    If \(f_1\) of matrix size \(d'\) is coefficient-reducible, there exists nonzero proper orthogonal projection \(P'\) so that \((I-P')f_1P=0\). 
    Let \(S'\) be the image of \(P\), then \(S'\neq\{0\}\) and \(S'\neq\mathbb{C}^{d'}\).
    Let \(V\) be the unitary change of basis to \(S'\oplus S'^\perp\). 
    Since \(S'^\perp=\image(I-P')\) and \((I-P')f_1P'=0\) then 
    \[
        V^*f_1V = \begin{pmatrix}
            f_3 & \star \\
            0 & f_4
        \end{pmatrix}
    \]
    so that 
    \[
        (V^*\oplus I_{d-d'})U^*fU(V\oplus I_{d-d'}) = \begin{pmatrix}
            f_3 & \star & \star \\
            0 & f_4 & \star \\
            0 & 0 & f_2
        \end{pmatrix}.
    \]
    Well \(U(V\oplus I_{d-d'})\) is unitary and \(f_3,f_4\) are of strictly smaller size than \(f_1\), so we can continue by induction.
\end{proof}

We have the conjugate-transpose \(^*\) on \(\M_n(\mathbb{C})\), which extends to an involution \(^*\) on \(\M_n(\mathbb{C})^g\).
We would like to internalise this involution from the domain space \(\M_n(\mathbb{C})^g\) to the variables \(\underline{x}\) themselves. 
Let \(\underline{y}=(y_1,\ldots,y_g)\) be a tuple of noncommuting variables, that also do not commute with \(x_1,\ldots,x_g\), and replace \(\mo\) with \(\mon\).
In other words, suppose we had an even number of noncommuting variables.
We define an antilinear involution \(^*\) on \(\mathbb{C}\mon\) as follows. 
On letters, \(x_i^*=y_i\) and \(y_i^*=x_i\).
On words, if \(w_1w_2\cdots w_k\in\mon\) then \((w_1w_2\cdots w_k)^*=w_k^*\cdots w_2^*w_1^*\).
Finally on noncommutative polynomials, if \(p\in\mathbb{C}\mon\) with \[p=\displaystyle\sum_{w\in\mon}p_w\,w,\quad\text{then}\quad p^*=\displaystyle\sum_{w\in\mon}\overline{p_w}\,w^*.\]
Matrix polynomials inherit an involution, also denoted by \(^*\), defined antilinearly as follows:
\[\M_d(\mathbb{C})\otimes_\mathbb{C}\mathbb{C}\mon\to\M_d(\mathbb{C})\otimes_\mathbb{C}\mathbb{C}\mon:A\otimes p\mapsto A^*\otimes p^*.\]
For each \(d,e\in\mathbb{N}\) we get an involution \({}^*\) from \(\M_{d\times e}(\mathbb{C})\otimes_\mathbb{C}\mathbb{C}\mon\) to \(\M_{e\times d}(\mathbb{C})\otimes_\mathbb{C}\mathbb{C}\mon\)), defined the same way. 

\begin{defn}
    For a polynomial \(f\in \M_d(\mathbb{C})\otimes_\mathbb{C}\mathbb{C}\mon\), we say that \(f\) is \textit{self-adjoint} if \(f^*=f\).
\end{defn}

When evaluating a self-adjoint polynomial at a point \((\underline{Z},\underline{Z}^*)\), we get a self-adjoint matrix, so the eigenvalues are all real.

\begin{defn}
    The \emph{real points} of \(\M_n(\mathbb{C})^g\times\M_n(\mathbb{C})^g\) are the tuples \((\underline{Z},\underline{Z}^*)\), denoted by 
    \[
        (\M_n(\mathbb{C})^g\times\M_n(\mathbb{C})^g)(\mathbb{R})\defeq\{(\underline{Z},\underline{Z}^*):\underline{Z}\in\M_n(\mathbb{C})^g\}.
    \]
\end{defn}
The real points inherit both Euclidean and Zariski topology from \(\M_n(\mathbb{C})^g\times\M_n(\mathbb{C})^g\).
This choice of notation will become clear in the next section. 

\begin{defn}
    The \emph{real free locus} of a noncommutative polynomial \(f\in\M_d(\mathbb{C})\otimes_\mathbb{C}\mathbb{C}\mon\) is defined to be 
    \[\mathscr{Z}^{\text{re}}(f)=\bigcup_{n\in\mathbb{N}}\mathscr{Z}_n^{\text{re}}(f),\quad\text{where}\quad\mathscr{Z}_n^{\text{re}}(f)=\{(\underline{X},\underline{X}^*)\in \M_n(\mathbb{C})^g\times \M_n(\mathbb{C})^g : \det f(\underline{X},\underline{X}^*)=0\}.\]
\end{defn}
Note that \(\mathscr{Z}_n^{\text{re}}(0)=(\M_n(\mathbb{C})^g\times\M_n(\mathbb{C})^g)(\mathbb{R})\).

\begin{defn}\label{semialg}
    For a self-adjoint matrix polynomial \(f\in \M_d(\mathbb{C})\otimes_\mathbb{C}\mathbb{C}\mon\), we define the \textit{positivity domain} of \(f\) to be the set 
    \[
        \mathcal{D}(f)=\bigcup_{n\in\mathbb{N}}\mathcal{D}_n(f),\quad\text{where}\quad\mathcal{D}_n(f)=\{(\underline{X},\underline{X}^*)\in \M_n(\mathbb{C})^g\times \M_n(\mathbb{C})^g : f(\underline{X},\underline{X}^*)\geq0\},
    \]
    where \(\geq\) denotes the Loewner order on \(\M_d(\mathbb{C})\otimes_\mathbb{C} \M_n(\mathbb{C})\cong\M_{dn}(\mathbb{C})\).
    The \emph{\(n\)-th level} \(\mathcal{D}_n(f)\) is a subset of \((\M_n(\mathbb{C})^g\times\M_n(\mathbb{C})^g)(\mathbb{R})\subset\M_n(\mathbb{C})^g\times\M_n(\mathbb{C})^g\), and inherits the corresponding subset topologies (both Euclidean and Zariski).
    We refer to the Euclidean boundary of \(\mathcal{D}_n(f)\) by \(\partial\mathcal{D}_n(f)\), and the Euclidean interior of \(\mathcal{D}_n(f)\) by \(\interior\mathcal{D}_n(f)\).
\end{defn}

We make the observation that if \(f\) is positive definite at a point \((\underline{X},\underline{X}^*)\), then this point is in the interior of the positivity domain, since positive definiteness is an open condition. 
That is,
\[
    \{(\underline{X},\underline{X}^*)\in\M_n(\mathbb{C})^g\times\M_n(\mathbb{C})^g: f(\underline{X},\underline{X}^*)\text{ is positive definite}\} \subset \interior\mathcal{D}_n(f),\quad\forall n\in\mathbb{N}.
\]
The converse is not necessarily true. 
\begin{exa}
    The first of the following two examples demonstrates that the interior of the positivity domain can be strictly larger than the positive definite points, and the second example shows that a positivity domain can have nonempty interior while also having no positive definite points.
    \begin{enumerate}
        \item[(i)] Consider the self-adjoint matrix polynomial \(f= (x_1+y_1)^2\in\mathbb{C}\Langle x_1,y_1\Rangle\).
        Its positivity domain is \(\mathcal{D}_n(f)=\{(Z,Z^*):Z\in\M_n(\mathbb{C})\}\), where \((0,0)\in\interior\mathcal{D}_n(f)\) but \(f(0,0)=0\) is not positive definite.
        \item[(ii)] Consider the self-adjoint matrix polynomial 
        \[
            \begin{pmatrix} (x_1+y_1)^2 & 0 \\ 0 & 0\end{pmatrix}\in\M_2(\mathbb{C}\Langle x_1,y_1\Rangle).
        \]
        Its positivity domain is \(\{(Z,Z^*):Z\in\M_n(\mathbb{C})\}\) but it is not positive definite anywhere.
    \end{enumerate}
\end{exa}

\begin{defn}
    Let \(f\in\M_d(\mathbb{C}\mon)\) be a matrix polynomial. 
    We say that \(f\) is \emph{positivity-minimal} if whenever \(\mathcal{D}(g)=\mathcal{D}(f)\), the size of \(g\) is at least \(d\).
\end{defn}

The following proposition shows that fullness is a sufficient condition for existence of positive definite points to be equivalent to having nonempty interior.

\begin{prop}\label{fullpos}
    Let \(f\in \M_d(\mathbb{C})\otimes_\mathbb{C}\mathbb{C}\mon\) be a full matrix polynomial.
    Then \(\mathcal{D}_n(f)\) has nonempty interior if and only if there exists \(\underline{Y}\in\M_n(\mathbb{C})^g\) for which \(f(\underline{Y},\underline{Y}^*)\) is positive definite.
\end{prop}
\begin{proof}
    Write \(U\) for the interior of \(\mathcal{D}_n(f)\subset(\M_n(\mathbb{C})^g\times\M_n(\mathbb{C})^g))(\mathbb{R})\).
    If \(f(\underline{Y},\underline{Y}^*)\) is positive definite for some \(\underline{Y}\in\M_n(\mathbb{C})^g\), then \((\underline{Y},
    \underline{Y}^*)\in U\) since positive definiteness is an open condition. 
    On the other hand, if \(U\neq\varnothing\) and the polynomial \(f_n=\det \restr{f}{\M_n(\mathbb{C})^g\times\M_n(\mathbb{C})^g}\) vanishes on \(U\), then 
    \[
        U \subset V(f_n)\cap(\M_n(\mathbb{C})^g\times\M_n(\mathbb{C})^g)(\mathbb{R})
    \] 
    where \(V(f_n)\) is the vanishing set of \(f_n\).
    Nonempty Euclidean open sets are Zariski dense, so the Zariski closure of \(U\) is \((\M_n(\mathbb{C})^g\times\M_n(\mathbb{C})^g)(\mathbb{R})\).
    Consequently, since \(V(f_n)\cap(\M_n(\mathbb{C})^g\times\M_n(\mathbb{C})^g)(\mathbb{R})\) is Zariski closed, then
    \[
        (\M_n(\mathbb{C})^g\times\M_n(\mathbb{C})^g)(\mathbb{R}) = V(f_n)\cap(\M_n(\mathbb{C})^g\times\M_n(\mathbb{C})^g)(\mathbb{R}) = \mathscr{Z}_n^{\text{re}}(f).
    \]
    By assumption, \((\underline{X},\underline{X}^*)\notin\mathscr{Z}_n^{\text{re}}(f)\) so we have a contradiction. 
    It follows that \(f_n\) does not vanish on all of \(U\), i.e., there exists \((\underline{Y},\underline{Y}^*)\in U\) for which \(f(\underline{Y},\underline{Y}^*)\) is nonsingular. 
    In other words, \(f(\underline{Y},\underline{Y}^*)\) is positive definite.
\end{proof}

It follows that, for full matrix polynomials \(f\), nonempty interiors propagate, i.e., if \(\interior\mathcal{D}_m(f)\neq\varnothing\) then \(\interior\mathcal{D}_{km}(f)\neq\varnothing\) for all \(k\), for all \(m\in n\mathbb{Z}^+\).
In particular if \(f(\underline{x},\underline{x}^*)\neq0\) for some \(\underline{x}\in\mathbb{C}^g\) and \(\interior\mathcal{D}_1(f)\neq\varnothing\), then \(\mathcal{D}_n(f)\) has nonempty interior for all \(n\).

The \emph{signature} of a Hermitian invertible matrix is a pair \((a,b)\in\mathbb{N}^2\) where \(a\) is the number of positive eigenvalues of the matrix, and \(b\) is the number of negative eigenvalues of the matrix.

\begin{prop}\label{pbdrylocus}
    Let \(f\in\M_d(\mathbb{C}\mon)\) be self-adjoint.
    Then \(\partial\mathcal{D}_n(f)\subset\mathscr{Z}_n^{\text{re}}(f)\) for all \(n\).
\end{prop}

\begin{proof}
    If \(\partial\mathcal{D}_n(f)\) is empty or \(\mathscr{Z}_n^{\text{re}}(f)=\left(\M_n(\mathbb{C})^g\times\M_n(\mathbb{C})^g\right)(\mathbb{R})\), we are done. 
    Suppose otherwise, and suppose toward contradiction that there is \((\underline{X},\underline{X}^*)\in \partial\mathcal{D}_n(f)\) but \(\det f(\underline{X},\underline{X}^*)\neq0\). 
    Let \(f_n\) be the polynomial \(\det f(-,-)\),
    so that \((\underline{X},\underline{X}^*)\) is in the complement of the variety in \(\M_n(\mathbb{C})^g\times\M_n(\mathbb{C})^g\) carved out by \(f_n\).
    This is a nonempty open set (by our assumption that \(\mathscr{Z}_n^{\text{re}}(f)\) is not the whole space), so there is an open connected neighbourhood \(U\) of \((\underline{X},\underline{X}^*)\) in \(\M_n(\mathbb{C})^g\times\M_n(\mathbb{C})^g\) on which \(f_n\) does not vanish.
    Let \(U'= U\cap\left(\M_n(\mathbb{C})^g\times\M_n(\mathbb{C})^g\right)(\mathbb{R})\), an open set in the subspace topology. 
    Consider the continuous map 
    \[
        U'\to\{\text{Invertible Hermitian matrices}\}:(\underline{Y},\underline{Y}^*)\mapsto f(\underline{Y},\underline{Y}^*).
    \]
    Since \(U'\) is connected, then its image under the above continuous map is connected. 
    By \cite[Corollary 1]{S66}, the space of invertible Hermitian matrices has connected components indexed by signature, with constant signature on each component. 
    So \(f(\underline{Y},\underline{Y}^*)\) has the same signature for all \((\underline{Y},\underline{Y}^*)\in U'\).
    Since \((\underline{X},\underline{X}^*)\) is not in the interior of \(\mathcal{D}_n(f)\) then \(f(\underline{X},\underline{X}^*)\) cannot be positive definite.
    So \(f(\underline{Y},\underline{Y}^*)\) is not positive semidefinite for all \((\underline{Y},\underline{Y}^*)\in U'\).
    Consequently, \(U'\cap\mathcal{D}_n(f)=\varnothing\), contradicting our assumption that \((\underline{X},\underline{X}^*)\) lies on the boundary of \(\mathcal{D}_n(f)\).
\end{proof}

\begin{defn}\label{unsig}
    Let \(f\in \M_d(\mathbb{C})\otimes_\mathbb{C}\mathbb{C}\mon\) be self-adjoint.
    We say \(f\) is \emph{unsignatured} (or \emph{unsignatured at level \(n\)}) if there exists \(n\in\mathbb{N}\) and \(\underline{X},\underline{Y}\in \M_n(\mathbb{C})^g\) for which \(f(\underline{X},\underline{X}^*)\) and \(f(\underline{Y},\underline{Y}^*)\) are invertible and have different signatures. 
\end{defn}

If \(f\) is unsignatured at level \(n\), there exists a tuple \(\underline{Y}\in \M_n(\mathbb{C})^g\) at which \(f(\underline{Y},\underline{Y}^*)\) is not positive semidefinite, so that \(\mathcal{D}_n(f)\) is a proper subset of \((\M_n(\mathbb{C})^g\times\M_n(\mathbb{C})^g)(\mathbb{R})\).
From \cite[Lemma 2.2]{HKV20}, unsignatured polynomials are full.

\begin{prop}\label{unsigunit}
    If \(f\) is unsignatured, then \(f\) is not a unit.
\end{prop}

\begin{proof}
    Suppose \(f\) were a unit, then \(f(\underline{X},\underline{X}^*)\) is invertible for all \(\underline{X}\in\M_n(\mathbb{C})^g\). 
    Since \(f\) is unsignatured, there exist \(\underline{Y},\underline{Z}\in\M_n(\mathbb{C})^g\) for which \(f(\underline{Y},\underline{Y}^*)\) and \(f(\underline{Z},\underline{Z}^*)\) are invertible with different signatures. 
    Let \(C=\left\{(1-t)\underline{Y}+t\underline{Z}:t\in[0,1]\right\}\), the line segment between \(\underline{Y}\) and \(\underline{Z}\). 
    This is a connected set in \(\M_n(\mathbb{C})^g\), so it has a connected image under the continuous map 
    \[
        \M_n(\mathbb{C})^g\to\{\text{Invertible Hermitian matrices}\}:\underline{X}\mapsto f(\underline{X},\underline{X}^*).
    \]
    By \cite[Corollary 1]{S66}, the space of invertible Hermitian matrices has \(n+1\) connected components, indexed by signature. 
    In particular, signature is locally constant on connected components; 
    so all elements of \(f(C)\) have the same signature. 
    This is a contradiction: \(f(\underline{Y},\underline{Y}^*)\) and \(f(\underline{Z},\underline{Z}^*)\) do not share the same signature.
\end{proof}

So unsignatured matrix polynomials are full nonunits. 
Unsignatured polynomials have well-behaved positivity domains, for if \(f\) were not unsignatured then either \(\mathcal{D}_n(f)\) is the complement of \(\mathscr{Z}_n^{\text{re}}(f)\) or \(\mathcal{D}_n(f)\) is contained in \(\mathscr{Z}^{\text{re}}_n(f)\) with empty interior, for each \(n\).
In either case, the dimension of the ambient space can be reduced. 
%HOW DO I REDUCE THE DIMENSION HERE? WHAT IS THE PROCEDURE?
We end this section by covering some convenient properties of unsignatured polynomials.

\begin{prop}\label{unsigpropagate}
    If \(f\) is unsignatured at level \(n\), then \(f\) is also unsignatured at level \(kn\) for each \(k\in\mathbb{N}\).
\end{prop}
\begin{proof}
    Suppose \(\underline{X},\underline{Y}\in\M_n(\mathbb{C})^g\) with \(f(\underline{X},\underline{X}^*)\) and \(f(\underline{Y},\underline{Y}^*)\) invertible with different signature. 
    Then \(f(\underline{I_k}\otimes \underline{X},\underline{I_k}\otimes \underline{X}^*)\) and \(f(\underline{I_k}\otimes \underline{Y},\underline{I_k}\otimes \underline{Y}^*)\) are both invertible and have different signature as well. 
\end{proof}

\begin{cor}\label{intpropagate}
    If \(f\in\M_d(\mathbb{C}\mon)\) is unsignatured at level \(n\) and \(\mathcal{D}_n(f)\) has nonempty interior, then \(\mathcal{D}_{mn}(f)\) has nonempty interior for all \(m\in \mathbb{N}_+\).
\end{cor}
\begin{proof}
    By Proposition \ref{fullpos} there exists \(\underline{X}\in\M_n(\mathbb{C})^g\) for which \(f(\underline{X},\underline{X}^*)\) is positive definite. 
    Consider \(I_m\otimes \underline{X}\in\M_{mn}(\mathbb{C})^g\) and apply Proposition \ref{fullpos} again.
\end{proof}

\begin{prop}\label{CIunit}
    If \(f\in\M_d(\mathbb{C}\mon)\) is an unsignatured atom and \(\mathcal{D}_1(f)\) has nonempty interior, then there exists an unsignatured coefficient-irreducible atom \(f'\in\M_{d'}(\mathbb{C}\mon)\) with \(d'\leq d\) for which \(\mathcal{D}(f') = \mathcal{D}(f)\).
    Further, there exists unitary \(U\in\M_d(\mathbb{C})\) and unit \(u\in\M_{d-d'}(\mathbb{C}\mon)\) so that \(U^*fU=f'\oplus u\).
    \end{prop}

\begin{proof}
    By Proposition \ref{CIfact} and using that \(f\) is self-adjoint, there exists coefficient-irreducible self-adjoint \(f_1,\ldots,f_r\) and a unitary \(U\in\M_d(\mathbb{C})\) so that 
    \[
        U^*fU = \begin{pmatrix}
            f_1 & 0 & \cdots & 0 \\
            0 & \ddots & \ddots & \vdots \\
            \vdots & \ddots & \ddots & 0 \\
            0 & \cdots & 0 & f_r
        \end{pmatrix}.
    \]
    Since \(f\) is an atom then \(U^*fU\) is an atom, so by Lemma \ref{diagatom}, exactly one of \(f_1,\ldots,f_r\) is an atom and the rest are units.
    Let \(f_j\) be the atom.
    By Proposition \ref{unsigunit}, each unit \(f_i\) (\(i\neq j\)) must be signatured, say with \(\alpha_i(n)\) many negative eigenvalues at each level \(n\).
    Since \(\mathcal{D}_1(f)\) has nonempty interior, \(\mathcal{D}_n(f)\) has nonempty interior for each \(n\).
    Then \(\alpha_i(n)=0\) for all \(i\neq j\), and so 
    \[
        \mathcal{D}_n(f)=\bigcap_{i\in\{1,\ldots,r\}}\mathcal{D}_n(f_i)=\mathcal{D}_n(f_j) \quad\forall n.
    \]
    So \(\mathcal{D}(f_j)=\mathcal{D}(f)\) where \(f_j\) is a coefficient-irreducible self-adjoint atom.
    Finally, \(f_j\) must be unsignatured since \(f\) is unsignatured and all \(f_i\), \(i\neq j\), are signatured.
\end{proof}

%%%%%%%%%%%%%%%%%%%%%%%%%%%%%%%%%%%%%%%%%%%%%%%%%%%%
% ALGEBRAIC GEOMETRY
%%%%%%%%%%%%%%%%%%%%%%%%%%%%%%%%%%%%%%%%%%%%%%%%%%%%
\section{\(\mathbb{R}\)-Varieties}\label{RVarieties}
Our aim in this section is to prove Theorem \ref{posdense}. 
To do this we need the theory developed in \cite[Chapter 2]{M20}. 
We will not provide the various definitions that come up in algebraic geometry (such as varieties, their dimensions, tangent spaces, etc.) 
These can be found in \cite{H77}.
Let \(\mathbb{A}^n\) refer to \(\mathbb{C}^n\).
We use \(\V(J)\) to denote the vanishing set of a subset of polynomials \(J\subset\mathbb{C}[x_1,\ldots,x_n]\), and \(\I(S)\) to denote the vanishing ideal of a subset of points \(S\subset \mathbb{A}^n\). 
We will also use \(V_\mathbb{R}(J)\) to denote the set of real points at which the polynomials in \(J\) vanish. 
The notation \(\mathcal{O}_X(U)\) will refer to the ring of regular functions on \(U\), where \(U\) is an open subset of \(X\).

We will refer to the dimension of a \(\K\)-variety \(X\) by \(\dim_\K(X)\).
By default, the absence of a subscript will imply that the field is \(\mathbb{C}\); that is, \(\dim\) refers to \(\dim_\mathbb{C}\). 
The dimension of a semialgebraic set (always over \(\mathbb{R}\)) will be denoted by \(\dim_{\text{Semi}}\).
We will call this the semialgebraic dimension.

\begin{defn}
    On \(\mathbb{A}^n\) we define a canonical involution map
    \[
        \sigma_n:\mathbb{A}^n\to\mathbb{A}^n:(x_1-a_1,\ldots,x_n-a_n)\mapsto(x_1-\overline{a_1},\ldots,x_n-\overline{a_n}).
    \]
\end{defn}
Notice that the set of real points \(\mathbb{R}^n\) sits inside of \(\mathbb{A}^n\) as the subset \(\{p\in\mathbb{A}^n:\sigma_n(p)=p\}\).
If \(f\in\mathbb{C}[x_1,\ldots,x_n]\) then \(\bar{f}=\sigma_1\circ f\circ \sigma_n\) if we consider \(f,\sigma_1,\sigma_n\) as functions on the closed points, where \(\bar{f}\in\mathbb{C}[x_1,\ldots,x_n]\) is the polynomial obtained from \(f\) by conjugating coefficients.

\begin{defn}
    If \(f\in\mathcal{O}_{\mathbb{A}^n}(U)\) for some open \(U\subset\mathbb{A}^n\), then the \emph{conjugate} of \(f\) is
    \(
        \bar{f} := \sigma_1\circ f\circ \sigma_n.
    \)
    For any subset \(X\subset\mathbb{A}^n\) we call the set \(\overline{X} := \{p\in\mathbb{A}^n:\sigma_n(p)\in X\}\) the \emph{conjugate of \(X\)}, and give it a Zariski topology via the subset topology in \(\mathbb{A}^n\).
\end{defn}

It is clear that, for any closed subset \(\V (f_1,\ldots,f_r) \subset \mathbb{A}^n\), the vanishing set of the conjugate functions is given by 
\(\V(\bar{f}_1,\ldots,\bar{f}_r)=\{p\in\mathbb{A}^n:\sigma_n(p)\in \V(f_1,\ldots,f_r)\}\); 
so \(\overline{\V(f_1,\ldots,f_r)}=\V(\bar{f}_1,\ldots,\bar{f}_r)\).
Hence \(\sigma_n\) is a continuous function on \(\mathbb{A}^n\); by involutivity, \(\sigma_n\) is a homeomorphism.

\begin{defn}
    Let \((X,\mathscr{O}_X)\) be a ringed space, where the functions map to \(\mathbb{C}\). 
    Then the \emph{anti-sheaf} \(\overline{\mathscr{O}_X}\) of \(\mathscr{O}_X\) is given by 
    \[
        \overline{\mathscr{O}_X}(U) = \{\sigma_1\circ f : f\in\mathscr{O}_X(U)\}.
    \]
\end{defn}
The anti-sheaf is indeed a sheaf of rings: multiplication happens pointwise in each ring, and gluability and uniqueness are inherited from \(\mathscr{O}_X\). 

\begin{defn}
    Let \((X,\mathscr{O}_X)\) be an algebraic \(\mathbb{C}\)-variety.
    The \emph{conjugate variety} of \(X\) is defined to be the ringed space \(\overline{X}:=(X,\overline{\mathscr{O}_X})\).
\end{defn}
If \(X\) is affine, its conjugate topological space \(\overline{X}\) is also affine and carries a structure sheaf obtained from restriction: \(\mathscr{O}_{\overline{X}}=\restr{\mathscr{O}_{\mathbb{A}^n}}{\overline{X}}\).
Then \(\sigma_n\) induces a ringed space isomorphism \((\overline{X},\mathscr{O}_{\overline{X}})\longrightarrow(X,\overline{\mathscr{O}_X})\).
In this way, we see that equipping the conjugate topological space of \(X\) with its restriction sheaf gives us the conjugate variety of \(X\). 
So in the affine case, conjugacy as a topological space is the same as conjugacy as a ringed space. 

\begin{defn}
    A map \(\varphi:X\to Y\) of varieties is \emph{anti-regular} if it is continuous and \(\sigma_1\circ f\circ\varphi\) is regular whenever \(f\) is regular.
\end{defn}

For example, \(\sigma_n\) on \(\mathbb{A}^n\) is anti-regular.
Generalising the maps \(\sigma_n\) to complex varieties, we say that an anti-regular involution \(\sigma\) on an algebraic \(\mathbb{C}\)-variety \(X\) is a \emph{real structure} on \(X\).
This construction is analogous to that of a \(C^*\)-algebra; in particular we will see that the fixed points of \(\sigma\) will form an important subspace of \(X\).
\begin{defn}
    A pair \((X,\sigma)\) where \(X\) is an algebraic \(\mathbb{C}\)-variety and \(\sigma\) a real structure on \(X\) is called an \emph{\(\mathbb{R}\)-variety}.
\end{defn}

Next we define morphisms of \(\mathbb{R}\)-varieties.
\begin{defn}
    A morphism of \(\mathbb{R}\)-varieties \((X,\sigma)\to(Y,\tau)\) is a morphism of varieties \(\varphi:X\to Y\) commuting with the real structures: 
    \[
        \begin{tikzcd}[row sep=huge, column sep=huge]
            X \arrow{r}{\sigma} \arrow[swap]{d}{\varphi} & X \arrow{d}{\varphi} \\
            Y \arrow{r}{\tau} & Y
        \end{tikzcd}
    \]
\end{defn}

These objects also have subobjects.
\begin{defn}
    A pair \((Y,\tau)\) is an \(\mathbb{R}\)-subvariety of \((X,\sigma)\) if \(Y\) is a subvariety of \(X\) and \(\tau=\restr{\sigma}{Y}\).
\end{defn}
If \(U\subset X\) open then \(\sigma(U)\) is open, so \(U\cap\sigma(U)\) is \(\sigma\)-invariant. 
Then \(\restr{\sigma}{U\cap\sigma(U)}\) is an involution on \(U\cap\sigma(U)\), and \((U\cap\sigma(U),\restr{\sigma}{U\cap\sigma(U)})\) is an \(\mathbb{R}\)-subvariety of \((X,\sigma)\).
In particular when \(X\) is irreducible and \(U\) is nonempty, \(U\cap\sigma(U)\) is a nonempty \(\mathbb{R}\)-subvariety.
\begin{defn}
    The \emph{real locus} of an \(\mathbb{R}\)-variety \((X,\sigma)\) is the set of points fixed by \(\sigma\):
    \[
        X(\mathbb{R})\defeq\{p\in X:\sigma(p)=p\}.
    \]
\end{defn}
Points in the real locus are called \emph{real points} of \(X\).

The map \(^*:\M_n(\mathbb{C})^g\to \M_n(\mathbb{C})^g:X=(X_1,\ldots,X_g)\mapsto X^*=(X_1^*,\ldots,X_g^*)\) defines a real structure on \(\M_n(\mathbb{C})^g\), with real points the \(g\)-tuples of \(n\times n\) Hermitian matrices over \(\mathbb{C}\).
The following example informs how these definitions interact with our free loci.
\begin{exa}\label{flrs}
    Let \(f\in\M_d(\mathbb{C})\otimes_\mathbb{C}\mathbb{C}\mon\) be a matrix polynomial, and \(\mathscr{X}=\mathscr{Z}_n(f)\) the \(n\)-graded piece of the free locus of \(f\).
    Then \(\mathscr{X}=V(f_n)\subset\mathbb{A}^{2gn^2}\) with anti-regular involution \(\sigma:\M_n(\mathbb{C})^g\times\M_n(\mathbb{C})^g:(X,Y)\mapsto(Y^*,X^*)\) defines a real structure on \(\mathbb{A}^{2gn^2}\).
    So \((\mathscr{X},\restr{\sigma}{\mathscr{X}})\) is an \(\mathbb{R}\)-variety.
    The real locus of \(\mathscr{X}\) is
    \begin{align*}
        \mathscr{X}(\mathbb{R}) &= \{(X,Y) \in \mathscr{X} : \sigma(X,Y)=(X,Y)\} = \mathscr{Z}_n^{\text{re}}(f).
    \end{align*}
\end{exa}

\begin{defn}
    Let \((X,\sigma)\) be an \(\mathbb{R}\)-variety. 
    We say that \(X\) \emph{has enough real points} if \(X(\mathbb{R})\) is Zariski-dense in \(X\).
\end{defn}
Having enough real points will be a desirable property of the \(\mathbb{R}\)-varieties we are interested in. 
Before finding sufficient conditions to ensure this property, we give a family of examples of \(\mathbb{R}\)-varieties with enough real points.
\begin{exa}
    If \(\sqrt{J}\subset\mathbb{R}[x_1,\ldots,x_n]\) a radical ideal and \(F=\V(\sqrt{J})\subset\mathbb{A}^n\) the associated complex affine variety, then \((F,\restr{\sigma_n}{F})\) is an \(\mathbb{R}\)-variety with enough real points if and only if \(\sqrt{J}\) is a real ideal (that is, \(a_1^2+\cdots+a_r^2\in\sqrt{J}\implies a_1,\ldots,a_r\in\sqrt{J}\).)
\end{exa}

Observe that a real structure \(\sigma\) on a variety \(X\) induces a group action of \(G\defeq\text{Gal}(\mathbb{C}/\mathbb{R})\) on each \(\mathscr{O}_X(U)\) by \(f\mapsto \sigma_1\circ f\circ \sigma\).
This induces an action on stalks, and further on cotangent spaces, which gives us the following Lemma \cite[Lemma 2.2.8]{M20}.
\begin{lem}\label{CotanDim}
    Let \((X,\sigma)\) be an \(\mathbb{R}\)-variety and let \(a\in X(\mathbb{R})\) be a real point of \(X\).
    Consider the cotangent space \(\mathfrak{m}_a/\mathfrak{m}_a^2\) at \(a\), where \(\mathfrak{m}_a\) is the maximal ideal of the stalk \(\mathcal{O}_{X,a}\) at \(a\).
    Then 
    \[
        \dim_\mathbb{C}\mathfrak{m}_a/\mathfrak{m}_a^2 = \dim_{\mathbb{R}}(\mathfrak{m}_a/\mathfrak{m}_a^2)^G.
    \]
\end{lem}
\begin{proof}
    If \([(f,U)]\in\mathfrak{m}_a/\mathfrak{m}_a^2\) then \(a\in U\) and \(f\in\mathcal{O}_X(U)\), but \(a\) is a real point so \(a\in\sigma^{-1}(U)\) for any representative in our germ.
    Thus \((\sigma_1\circ f\circ \sigma,\sigma^{-1}(U))\in\mathfrak{m}_a/\mathfrak{m}_a^2\), which gives an antilinear involution on \(\mathfrak{m}_a/\mathfrak{m}_a^2\) with fixed points \((\mathfrak{m}_a/\mathfrak{m}_a^2)^G\).
    Now for a finite-dimensional \(\mathbb{C}\)-vector space \(E\) with antilinear involution map \(\rho\) and fixed ponts \(E^\rho\), an \(\mathbb{R}\)-basis of \(E^\rho\) will also be a \(\mathbb{C}\)-basis of \(E\) (not hard to see: \(\rho\) is a diagonalisable matrix on the \(\mathbb{R}\)-vector space \(E\) with eigenvalues \(1\) and \(-1\), so decompose the \(\mathbb{R}\)-vector space \(E = E^\rho\oplus\ker(\rho+\text{id.})\)). 
    Thus \(\mathfrak{m}_a/\mathfrak{m}_a^2\) has a \(\mathbb{C}\)-vector space basis consisting of all \(\sigma\)-invariant functions.
\end{proof}

The next proposition, and the corollary following, are well-known \cite[Theorem 2.2.9]{M20}. 
We write down the proof here for clarity and completeness.
\begin{prop}\label{density}
    Let \((X,\sigma)\) be an \(\mathbb{R}\)-variety.
    \begin{itemize}
        \item[1.] Suppose \(X\subset\mathbb{A}^n\) an irreducible affine \(\mathbb{C}\)-variety such that \(I(X)\) is generated by polynomials with real coefficients, and \(\sigma=\restr{\sigma_n}{X}\).
        Then the real locus \(X(\mathbb{R})\) is Zariski dense in \(X\) if and only if \(X(\mathbb{R})\) contains a regular point of \(X\).
        \item[2.] Now suppose \((X,\sigma)\) is any \(\mathbb{R}\)-variety. 
        Then the real locus \(X(\mathbb{R})\) is Zariski dense in every irreducible component of \(X\) containing a regular real point of \(X\) if and only if \(X(\mathbb{R})\) contains a regular real point of \(X\).
    \end{itemize}
\end{prop}

\begin{proof}
    We begin with statement 1 as it is required for our proof of statement 2.
    \begin{itemize}
        \item[1.] Let \(V_\mathbb{R}(S)\) be the set of prime ideals in \(\mathbb{R}[x_1,\ldots,x_n]\) containing the set \(S\), for any subset \(S\) of \(\mathbb{R}[x_1,\ldots,x_n]\).  
        Since \(X\) is irreducible, \(J=I(X)\) is prime in \(\mathbb{C}[x_1,\ldots,x_n]\). 
        Pulling back along the inclusion \(\mathbb{R}[x_1,\ldots,x_n]\hookrightarrow\mathbb{C}[x_1,\ldots,x_n]\), we get a prime ideal \(J_\mathbb{R} \defeq J\cap\mathbb{R}[x_1,\ldots,x_n]\) of \(\mathbb{R}[x_1,\ldots,x_n]\).
        But \(J\) is generated by polynomials with real coefficients so \(X=V(J_\mathbb{R})\) and \(X(\mathbb{R})=V_\mathbb{R}(J_\mathbb{R})\). 
        Let \(d\) be the dimension of \(X\), i.e., the Krull dimension of \(\mathbb{C}[x_1,\ldots,x_n]/J\).
        Then \(\mathbb{R}[x_1,\ldots,x_n]/J_\mathbb{R}\) also has dimension \(d\) (follows from the injective integral map \(\mathbb{R}[x_1,\ldots,x_n]/J_\mathbb{R}\hookrightarrow\mathbb{C}[x_1,\ldots,x_n]/J\).)
        
        Since the regularity locus of an irreducible variety is nonempty open, it is dense, and so we only need to show that \(X(\mathbb{R})\) is dense in \(X\) given \(X(\mathbb{R})\) contains a regular point of \(X\).
        Let \(a\in X(\mathbb{R})\) be such a regular point. 
        Then the dimension of the cotangent space of \(X\) at \(a\) is equal to \(d\), so there is a Euclidean open neighbourhood \(W\subset X\) of \(a\) for which \(W\) is a differentiable manifold of real dimension \(2d\). 
        Further \(W\cap X(\mathbb{R})\) is a Euclidean neighbourhood of \(a\) in \(X(\mathbb{R})\) of real dimension \(d\). 
        So \(X(\mathbb{R})\) is a real variety of dimension \(d\), thus there exists a length \(d\) chain of prime ideals in \(\mathbb{R}[x_1,\ldots,x_n]/I(X(\mathbb{R}))\). 
        Notice that \(J_\mathbb{R}\subset I(X(\mathbb{R}))\); if \(J_\mathbb{R}\neq I(X(\mathbb{R}))\) then we would get a length \(d+1\) chain of prime ideals in \(\mathbb{R}[x_1,\ldots,x_n]/J_\mathbb{R}\), a contradiction.
        So \(I(X(\mathbb{R}))=J_\mathbb{R}\), consequently \[\overline{X(\mathbb{R})}=V(I(X(\mathbb{R})))=V(J_\mathbb{R})=X.\]

        \item[2.] Let \(Y\subset X\) be the regular locus of \(X\).
        Suppose first that \(X(\mathbb{R})\) is dense in every irreducible component of \(X\) containing a regular real point.
        If \(Z\) is an irreducible component containing a regular real point of \(X\) then \(Z\cap Y\) is a nonempty open subset of \(Z\), and by hypothesis \(X\cap X(\mathbb{R}\) is dense in \(Z\). 
        So \(X(\mathbb{R})\cap Z\cap Y)\) is nonempty, i.e. \(X(\mathbb{R})\) contains a  regular point of \(X\).

        Now suppose \(X(\mathbb{R})\) contains a regular point of \(X\).
        Let \(Z\) be an irreducible component of \(X\) containing a regular real point of \(X\), and let \(a\in X(\mathbb{R})\cap Y\cap Z\) be this point.
        Since \(Z\) is a \(\mathbb{C}\)-variety, there exists an affine open \(U\subset Z\) containing \(a\).
        By irreducibility of \(Z\), we have that \(U'=U\cap\sigma(U)\) is a nonempty \(\mathbb{R}\)-subvariety of \(Z\), and \(a\in U'\).
        Thus \((U',\restr{\sigma}{U'})\) is isomorphic to an \(\mathbb{R}\)-subvariety \((V,\restr{\sigma_n}{V})\) of \((\mathbb{A}^n,\sigma_n)\), call this map \(\varphi:U'\to V\).
        Then \(\varphi\circ\restr{\sigma}{U'}=\sigma_n\circ\varphi\), so for any \(p\in U'(\mathbb{R})\) we get \(\varphi(p)=\sigma_n(\varphi(p))\), which is fixed by \(\sigma_n\).
        So \(U'(\mathbb{R})=\varphi^{-1}(V(\mathbb{R}))\).
        Applying Part 1, we get that \(V(\mathbb{R})\) is dense in \(V\), so \(\overline{U'(\mathbb{R})}=\overline{\varphi^{-1}(V(\mathbb{R})}=\varphi^{-1}(V)=U'\).
        Finally, \(U'(\mathbb{R})=U'\cap X(\mathbb{R})\) and \(U'\) is dense in \(Z\), so \(X(\mathbb{R})\cap Z\) is dense in \(Z\).
    \end{itemize}
\end{proof}

\begin{cor}\label{irredense}
    Suppose \((X,\sigma)\) is an irreducible \(\mathbb{R}\)-variety.
    If \(X(\mathbb{R})\) contains a regular point of \(X\), then \(X\) has enough real points.
\end{cor}

We now relate these results back to free loci. The following two lemmas appear in \cite{HKV18}, we include the proofs here for ease.
\begin{lem}\label{realhyp}
    Let \(f\in\M_d(\mathbb{C})\otimes_\mathbb{C}\mathbb{C}\mon\) be a noncommutative polynomial, \(\mathscr{X}=\mathscr{Z}_n(f)\) and \(\sigma\) the real structure on \(\M_n(\mathbb{C})^g\times\M_n(\mathbb{C})^g\) defined in example \ref{flrs}. 
    If \(f\) is unsignatured at level \(n\) and \(\mathcal{D}_n(f)\) has nonempty interior, then \(\overline{\mathscr{X}(\mathbb{R})}\) has codimension 1 in \(\M_n(\mathbb{C})^g\times\M_n(\mathbb{C})^g\).
\end{lem}
\begin{proof}
    Let \(\mathbb{A}=\left(\M_n(\mathbb{C})^g\times\M_n(\mathbb{C})^g\right)(\mathbb{R})=\{(Z,Z^*):Z\in\M_n(\mathbb{C})^g\}\).
    The real points only form a real vector space, with dimension \(2gn^2\). 
    We take the real part of the complex inner product on the whole space \(\M_n(\mathbb{C})^g\times\M_n(\mathbb{C})^g\), i.e.
    \[
        \langle (U,U^*),(V,V^*)\rangle_\mathbb{R}=\text{Re}\sum_{j=1}^{g}\trace(U_jV_j^*)\quad\forall (U,U^*),(V,V^*)\in\mathbb{A}.
    \]
    By proposition \ref{fullpos} there exists \(X\in\M_n(\mathbb{C})^g\) for which \(f(X,X^*)\) is positive definite. 
    Since \(f\) is unsignatured, there exists \(Y\in\M_n(\mathbb{C})^g\) for which \(f(Y,Y^*)\) is nonsingular and not positive semidefinite.
    This is a Euclidean-open condition; there is \(\varepsilon>0\) such that the ball \(\mathcal{B}\) of radius \(\varepsilon\) around \((Y,Y^*)\) in \(\mathbb{A}\) does not intersect \(\mathscr{Z}_n(f)\) nor \(\mathcal{D}_n(f)\).
    Consider the plane \(P\) in \(\mathbb{A}\) normal to \((X,X^*)-(Y,Y^*)\), and let \(D\) be the intersection of \(P\) with \(\mathcal{B}\):
    \[
        D = \left\{(Z,Z^*)\in\M_n(\mathbb{C})^g\times\M_n(\mathbb{C})^g:\norm{Z-Y}<\varepsilon,\,\sum_{j}\trace\left((Z_j-Y_j)^*(X_j-Y_j)\right)=0\right\}.
    \]
    Then \(D\) is semialgebraic in \(\mathbb{A}\) with real dimension \(2gn^2-1\) (there is a \((2gn^2-1)\)-dimensional region in \(D\) around the point \((Y,Y^*)\)).
    Let \(\mathcal{C}\) be the convex hull of \(\{(X,X^*)\}\cap D\); this cone is a collection of line segments from \((X,X^*)\) to a distinct point in \(D\). 
    In particular, \(\mathscr{X}(\mathbb{R})\) intersects each line segment not at an endpoint.
    Consider the projection from \(\mathscr{X}(\mathbb{R})\cap\mathcal{C}\) onto \(D\) along each respective line segment: \[\pi:\mathscr{X}(\mathbb{R})\cap\mathcal{C}\twoheadrightarrow D.\]
    Then \(\pi\) is a semialgebraic map, so by \cite{BCR98} Theorem 2.8.8, \(\dim_{\text{Semi}}(\mathscr{X}(\mathbb{R})\cap\mathcal{C})=\dim_{\text{Semi}} D=2gn^2-1\).
    Well \(\mathscr{X}(\mathbb{R})\cap\mathcal{C}\subset\overline{\mathscr{X}(\mathbb{R})}\) so \(\overline{\mathscr{X}(\mathbb{R})}\) has a real point \(p\) at which the real dimension of the cotangent space is \(2gn^2-1\). 
    Thus by \cite{BCR98} Theorem 2.8.2, \(\dim\overline{\mathscr{X}(\mathbb{R})}\geq2gn^2-1\).
    Since \(\mathscr{X}\neq\M_n(\mathbb{C})^g\times\M_n(\mathbb{C})^g\) (because \(f\) is unsignatured) then \(\dim\overline{\mathscr{X}(\mathbb{R})}=2gn^2-1\).
\end{proof}
% I WANT TO USE THE COTANGENT DIMENSION LEMMA IN THE ABOVE PROOF...

\begin{lem}\label{flerp}
    Let \(f\in\M_d(\mathbb{C})\otimes_\mathbb{C}\mathbb{C}\mon\) be an unsignatured atom, and suppose \(\mathcal{D}_m(f)\) has nonempty interior for some \(m\in\mathbb{N}\). 
    Then there exists \(n_0\in\mathbb{N}\) for which \(\mathscr{Z}_n(f)\) has enough real points whenever \(n\) is a positive multiple of \(n_0\).
\end{lem}
\begin{proof}
    Since \(f\) is unsignatured, there is \(r\in\mathbb{N}\) for which \(f\) is unsignatured at level \(r\).
    By Proposition \ref{unsigpropagate} \(f\) is unsignatured at level \(\ell mr\) for all \(\ell\in\mathbb{N}\), and by Corollary \ref{intpropagate} \(\mathcal{D}_{\ell mr}(f)\) has nonempty interior for all \(\ell\in\mathbb{N}\). 
    By \cite[Theorem 2.9]{HKV20}, there exists \(n_0'\) for which \(\mathscr{Z}_n(f)\) is irreducible for each \(n\geq n_0'\).  
    Choose \(\ell\) large enough so that \(\ell mr\geq n_0'\), and set \(n_0\) to be the first \(\ell mr\) which is larger than \(n_0'\).
    Let \(n\) be a positive multiple of \(n_0\), so that \(f\) is unsignatured at level \(n\) and \(\mathcal{D}_n(f)\) has nonempty interior and \(\mathscr{Z}_n(f)\) is irreducible.
    
    Let \(\mathscr{X}=\mathscr{Z}_n(f)\).
    The singularity locus of \(\mathscr{X}\) has codimension at least 2. 
    By Lemma \ref{realhyp} \(\overline{\mathscr{X}(\mathbb{R})}\) has codimension 1, so \(\mathscr{X}(\mathbb{R})\) has semialgebraic dimension \(2gn^2-1\); it cannot be contained in the singularity locus of \(\mathscr{X}\).
    Then \(\mathscr{X}(\mathbb{R})\) contains a regular point of \(\mathscr{X}\).
    By Corollary \ref{irredense}, \(\mathscr{X}(\mathbb{R})\) is Zariski dense in \(\mathscr{X}\).
\end{proof}

The following result gives a connection between the positivity domain and the free locus.

\begin{thm}\label{posdense}
    Suppose \(f=f_1\oplus \cdots\oplus f_r\in\M_d(\mathbb{C})\otimes_\mathbb{C}\mathbb{C}\mon\) is unsignatured and positivity-minimal, where \(f_1,\ldots,f_r\) are atoms, and suppose \(\mathcal{D}_m(f)\) has nonempty interior for some \(m\in\mathbb{N}\).
    Then there exists \(n_0\in\mathbb{N}\) for which the Eucilidean boundary \(\partial\mathcal{D}_n(f)\) of \(\mathcal{D}_n(f)\) is Zariski dense in \(\mathscr{Z}_n(f)\) whenever \(n\) is a positive multiple of \(n_0\).
\end{thm}
\begin{proof}
    Observe that
    \[
        \mathcal{D}_n(f)=\bigcap_{i\in\{1,\ldots,r\}}\mathcal{D}_n(f_i)\quad\text{and}\quad\mathscr{Z}_n(f)=\bigcup_{i\in\{1,\ldots,r\}}\mathscr{Z}_n(f_i),\quad \forall n\in\mathbb{N}.
    \]
    Choose \(n_0\) large enough and a multiple of \(m\) so that \(f\) is unsignatured at level \(n_0\), \(\mathcal{D}_{n_0}(f)\) has nonempty interior, and each \(\mathscr{Z}_{n_0}(f_i)\) is an irreducible hypersurface. 
    Fix \(n\) a positive multiple of \(n_0\).
    Assume that at level \(n\), no \(f_i\) is positive semidefinite everywhere (if some \(f_i\) is, discard \(\mathcal{D}_n(f_i)\) from the intersection above).
    Also assume no \(f_i\) vanishes nowhere (if some \(f_i\) does, then discard \(\mathscr{Z}_n(f_i)\) from the union above).

    Since \(\mathcal{D}_n(f)\) has nonempty interior, there exists \((X,X^*)\) in the interior of \(\mathcal{D}_n(f)\); then \((X,X^*)\) is in the interior of each \(\mathcal{D}_n(f_i)\). 
    Fix \(j\in\{1,\ldots,r\}\).
    If \(\bigcap_{i\in\{1,\ldots,r\}\backslash\{j\}}\mathcal{D}_n(f_i)\subset\mathcal{D}_n(f_j)\) then \(f\) would not be minimal; so there exists \((Y,Y^*)\in\mathcal{D}_n(f_i)\) for all \(i\neq j\) and \((Y,Y^*)\notin\mathcal{D}_n(f_j)\).
    Let \(\mathbb{A}=\left(\M_n(\mathbb{C})^g\times\M_n(\mathbb{C})^g\right)(\mathbb{R})\) with the real inner product
    \[
        \langle(U,U^*),(V,V^*)\rangle_{\mathbb{R}}=\text{Re}\sum_{j=1}^{g}\trace(U_jV_j^*)\quad\forall(U,U^*),(V,V^*)\in\mathbb{A}.
    \]
    Let \(\mathcal{B}\subset\mathbb{A}\) be the ball of radius \(\varepsilon>0\) around \((Y,Y^*)\) so that \((Z,Z^*)\in\mathcal{D}_n(f_i)\) for all \(i\neq j\) and \((Z,Z^*)\notin\mathcal{D}_n(f_j)\), for all \((Z,Z^*)\in\mathcal{B}\).
    Let \(D\) be the intersection of \(\mathcal{B}\) with the plane in \(\mathbb{A}\) orthogonal to \((Y,Y^*)-(X,X^*)\):
    \[
        D:=\left\{(Z,Z^*)\in\mathbb{A}:\norm{Z-Y}<\varepsilon,\,\sum_j\trace\left((Z_j-Y_j)^*(X_j-Y_j)\right)=0\right\}.
    \]
    Let \(\mathcal{C}\) be the convex hull of \(\{(X_0,X_0^*)\}\cup D\), then \(\mathcal{C}\) consists of line segments from \((X_0,X_0^*)\) into \(D\).
    Denote each line segment by 
    \[
        \ell_Z:=\{(1-t)(X,X^*)+t(Z,Z^*):t\in[0,1]\}\quad\forall (Z,Z^*)\in D,
    \]
    so that \(\mathcal{C}=\bigcup_{(Z,Z^*)\in D}\ell_Z\).
    Each line segment intersects \(\partial\mathcal{D}_n(f_j)\cap\partial\mathcal{D}_n(f)\) at least once (positive semidefiniteness is a continuous condition) and at most a finite number of times (both \(\partial\mathcal{D}_n(f_j)\) and \(\partial\mathcal{D}_n(f)\) lie on closed varieties), so we get a bijection
    \[
        \rho:\{\ell_Z:(Z,Z^*)\in D\}\xlongrightarrow{\sim}\partial\mathcal{D}_n(f_j)\cap\mathcal{D}_n(f)\cap\mathcal{C}:\ell_Z\mapsto\inf\partial\mathcal{D}_n(f_j)\cap\partial\mathcal{D}_n(f)\cap\ell_Z.
    \]
    Thus we have a bijection 
    \[
        D\xlongrightarrow{\sim}\partial\mathcal{D}_n(f_j)\cap\partial\mathcal{D}_n(f)\cap\mathcal{C}:Z\mapsto\rho(\ell_Z)
    \]
    which is semialgebraic (projections and \(\inf\) can be written as first-order formulas in the language of ordered fields with parameters in \(\mathbb{R}\) \cite[Proposition 2.2.4]{BCR98}.)
    It follows from \cite[Theorem 2.8.8]{BCR98} that \(\dim_{\text{Semi}}\left(\partial\mathcal{D}_n(f_j)\cap\partial\mathcal{D}_n(f)\cap\mathcal{C}\right)=\dim_{\text{Semi}}D=2gn^2-1\).
    Now by \cite[Theorem 2.8.2]{BCR98}, \(\dim\overline{\partial\mathcal{D}_n(f_j)\cap\partial\mathcal{D}_n(f)\cap\mathcal{C}}=2gn^2-1\).
    Consequently \(\overline{\partial\mathcal{D}_n(f_j)\cap\partial\mathcal{D}_n(f)\cap\mathcal{C}}\) is a hypersurface inside the irreducible hypersurface \(\mathscr{Z}_n(f_j)\).
    Thus \(\overline{\partial\mathcal{D}_n(f_j)\cap\partial\mathcal{D}_n(f)}=\mathscr{Z}_n(f_j)\).
    This holds for each \(j\in\{1,\ldots,r\}\) so 
    \[
       \overline{\bigcup_{j\in\{1,\ldots,r\}}\left(\partial\mathcal{D}_n(f_j)\cap\partial\mathcal{D}_n(f)\right)}=\bigcup_{j\in\{1,\ldots,r\}}\overline{\left(\partial\mathcal{D}_n(f_j)\cap\partial\mathcal{D}_n(f)\right)}=\bigcup_{j=\{1,\ldots,r\}}\mathscr{Z}_n(f_j)=\mathscr{Z}_n(f).
    \]
    Each \(\partial\mathcal{D}_n(f_j)\cap\partial\mathcal{D}_n(f)\) is contained in \(\partial\mathcal{D}_n(f)\), so we are done.
\end{proof}

%%%%%%%%%%%%%%%%%%%%%%%%%%%%%%%%%%%%%%%%%%%%%%%%%%%%
% POSITIVITY DOMAINS OF QUADRATICS
%%%%%%%%%%%%%%%%%%%%%%%%%%%%%%%%%%%%%%%%%%%%%%%%%%%%
\section{Positivity Domains of Quadratic Polynomials}\label{PosDomQuadPoly}
Set \(\K=\mathbb{C}\).
We use the notation \([n]\defeq\{1,2,\ldots,n\}\).
We investigate the class of \emph{hereditary} degree 2 matrix polynomials 
\[
    f=A_0+\sum_{i\in[g]}(A_ix_i+B_iy_i)+\sum_{i,j\in[g]}C_{i,j}y_ix_j
\]
where \(A_0,A_i,B_i,C_{i,j}\in\M_d(\mathbb{C})\) for all \(i,j\in[g]\) (first appeared in hereditary functional calculus \cite{Agler90}.) 
Write \(L=A_0+\sum_{i\in[g]}(A_ix_i+B_iy_i)\) for the degree 1 piece of \(f\).
Suppose also that \(f\) is self-adjoint, then \(A_0^*=A_0\), \(A_i^*=B_i\) and \(C_{i,j}^* = C_{j,i}\) for all \(i,j\in[g]\). 
It follows that \(L\) is a self-adjoint linear pencil, and
\[
    C= \begin{pmatrix}
        C_{1,1}&\cdots&C_{1,g}\\
        \vdots&\ddots&\vdots\\ 
        C_{g,1}&\cdots&C_{g,g}
    \end{pmatrix}\in\M_{dg}(\mathbb{C})
\] 
is a Hermitian matrix.
We can write \(C=\hat{P}-\hat{Q}\) where \(\hat{P},\hat{Q}\) are positive semidefinite.
Write \(\hat{P}=P^*P\) and \(\hat{Q}=Q^*Q\) where \(P=\sqrt{\hat{P}}\) and \(Q=\sqrt{\hat{Q}}\).
Decompose \(P\) and \(Q\) each into \(g\) many \((dg\times d)\)-blocks, i.e. \(P=\begin{pmatrix}H_1&H_2&\cdots&H_g\end{pmatrix}\) and \(Q=\begin{pmatrix}K_1&K_2&\cdots&K_g\end{pmatrix}\) where \(H_i,K_i\in \M_{dg\times d}(\mathbb{C})\) for all \(i\in[g]\).
It follows that 
\[
    \sum_{i,j\in[g]}C_{i,j}y_ix_j = \left(\sum_{i\in[g]}H_ix_i\right)^*\left(\sum_{i\in[g]}H_ix_i\right)-\left(\sum_{i\in[g]}K_ix_i\right)^*\left(\sum_{i\in[g]}K_ix_i\right).
\]
Setting \(H=\sum_{i\in[g]}H_ix_i\in\M_{dg\times d}(\mathbb{C})\otimes\mathbb{C}\mon\) and \(K=\sum_{i\in[g]}K_ix_i\in\M_{dg\times d}(\mathbb{C})\otimes\mathbb{C}\mon\), we see that our self-adjoint hereditary quadratic \(f\) can be written as \[f=L+H^*H-K^*K.\]
In other words, self-adjoint hereditary quadratic polynomials are the same data as a certain subset of triples of linear pencils \((L,H,K)\).

\begin{defn}
    We say that a triple of linear pencils \((L,H,K)\) is a \emph{line system} if
    \begin{align*}
        L &= L_0 + \sum_{i\in[g]}L_ix_i + \sum_{i\in[g]}L_i^*y_i \in \M_d(\mathbb{C})\otimes\mathbb{C}\mon, \\
        H &= \sum_{i\in[g]}H_ix_i \in \M_{dg\times d}(\mathbb{C}) \otimes \mathbb{C}\mon, \\
        K &= \sum_{i\in[g]}K_ix_i \in \M_{dg\times d}(\mathbb{C}) \otimes \mathbb{C}\mon
    \end{align*}
    where \(L_0\in\M_d(\mathbb{C})\) is Hermitian, \(L_i\in\M_d(\mathbb{C})\) for all \(i\in[g]\), and \(H_i,K_i\in\M_{dg\times d}(\mathbb{C})\) for all \(i\in[g]\).
    The polynomial \(f=L+H^*H-K^*K\) will be called the \emph{associated polynomial} to the line system \((L,H,K)\).
\end{defn}

We aim to study the positivity domains of these hereditary quadratic polynomials.
While positivity domains of linear pencils are convex, the positivity domains of hereditary quadratics may not be. 
In fact, their components may not even be simply-connected.
\begin{exa}\label{annulus}
    Consider the polynomial in \(\M_2(\mathbb{C})\otimes_\mathbb{C}\mathbb{C}\Langle x,y\Rangle\) given by
    \[
        f = 
        \begin{pmatrix}
            -1 & 0 \\
            0 & 4
        \end{pmatrix} + \begin{pmatrix}
            1 & 0 \\
            0 & -1
        \end{pmatrix}\otimes yx
        =
        \begin{pmatrix}
            yx-1 & 0 \\
            0 & 4-yx
        \end{pmatrix}.
    \]
    At level 1, the positivity domain is 
    \begin{align*}
        \mathcal{D}_1(f) &= \left\{z\in\mathbb{C}:\begin{pmatrix}\abs{z}^2-1&0\\0&4-\abs{z}^2\end{pmatrix}\text{ is positive semidefinite}\right\} \\
        &= \left\{z\in\mathbb{C}:\abs{z}^2-1\geq0,\,4-\abs{z}^2\geq0\right\}.
    \end{align*}
    This is an annulus with inner radius \(1\) and outer radius \(2\):
    \begin{center}
        \begin{tikzpicture}[scale=0.8]
            \draw [name path=B, thick, draw,fill=blue!20](0,0) circle (2cm); 
            \draw [name path=A, thick, draw,fill=white](0,0) circle (1cm); 
            \draw[->] (-3,0) -- (3,0) node[right] {\(x\)};
            \draw[->] (0,-3) -- (0,3) node[above right] {\(y\)};
            \draw (0,0) node[below left] {\(0\)};
            \draw (1,0) node[below left] {\(1\)};
            \draw (2,0) node[below left] {\(2\)};
        \end{tikzpicture}
    \end{center}
\end{exa}
The following is a slightly more exotic example, computed with  Mathematica \cite{Mathematica}.
\begin{exa}
    Consider the family of polynomials
    \[
        \mathbb{R}\to\M_2(\mathbb{C})\otimes_\mathbb{C}\mathbb{C}\Langle x,y\Rangle:t\mapsto f_t = \begin{pmatrix}
            yx-1 & 1+x \\
            1+y & t - yx
        \end{pmatrix}.
    \]
    The family of domains \(t\mapsto\mathcal{D}_1(f_t)\) is simply-connected with smooth boundary for \(t\in(0,\alpha)\) where \(\alpha\) is the unique real solution to \(\alpha^3-5\alpha^2-10\alpha-11=0\), i.e., \(\alpha\approx6.729\).
    At \(t=\alpha\), the boundary has a singularity at \(\beta+0i\) where \(\beta\) is the unique real solution to \(\beta^3-2\beta^2+\beta-1=0\), i.e., \(\beta\approx1.755\).
    For \(t>\alpha\), the domain is no longer simply-connected.
    Below are sketches of \(\mathcal{D}_1(f_t)\) for varying \(t\) values, where the horizontal and vertical axes represent the real and imaginary parts respectively.
    \begin{figure}[!ht]
        \begin{center}
            \begin{minipage}[b]{0.32\textwidth}
                \includegraphics[width=\textwidth]{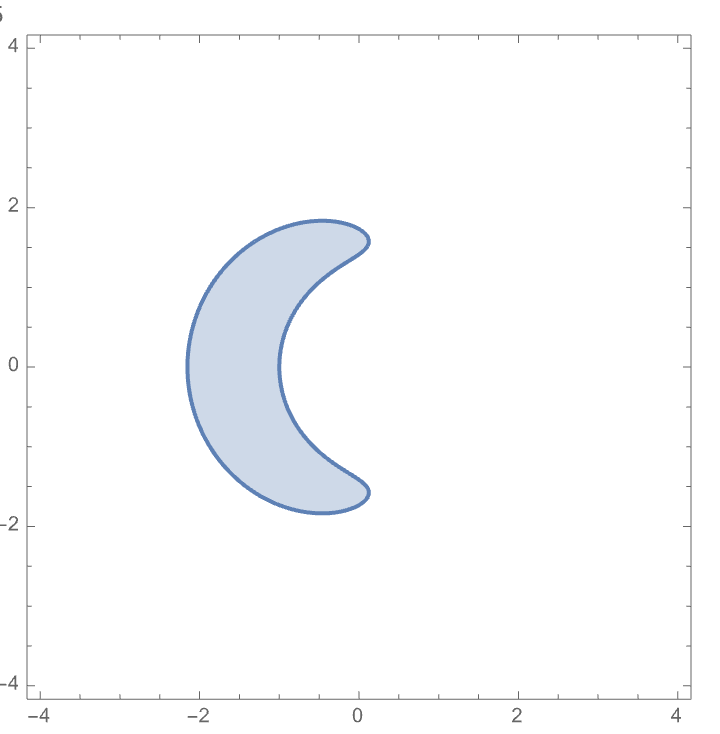}
                \caption{\(t=5\).}
            \end{minipage}
            \hfill
            \begin{minipage}[b]{0.32\textwidth}
                \includegraphics[width=\textwidth]{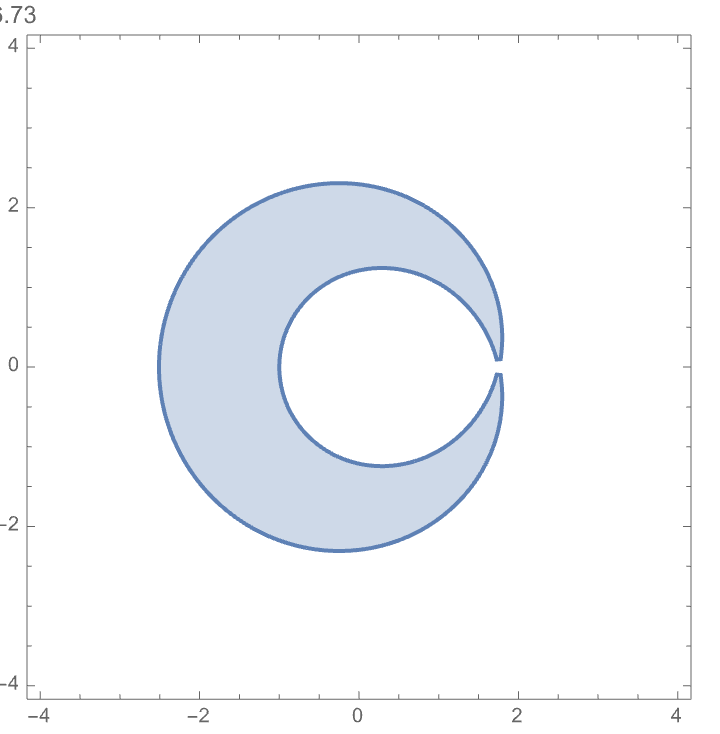}
                \caption{\(t=\alpha\)}
            \end{minipage}
            \hfill
            \begin{minipage}[b]{0.32\textwidth}
                \includegraphics[width=\textwidth]{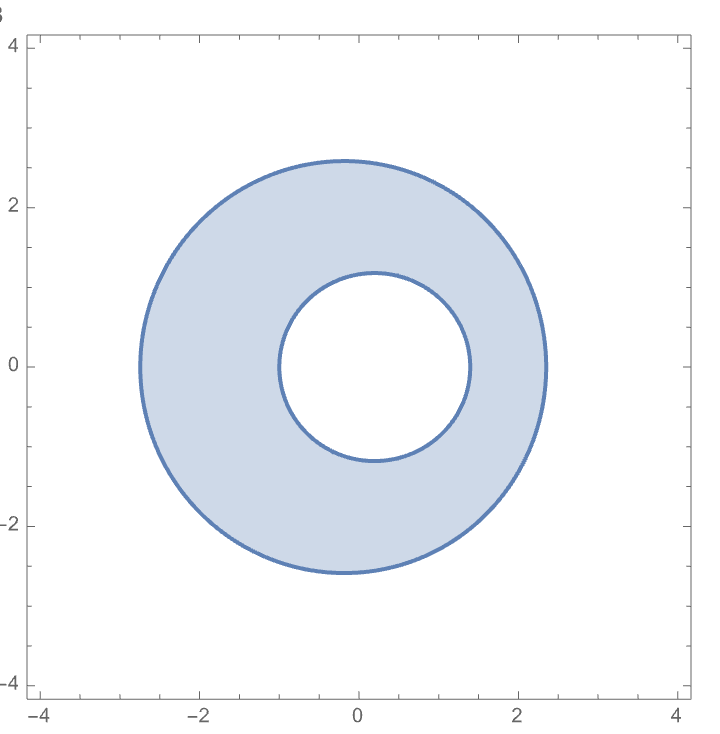}
                \caption{\(t=8\)}
            \end{minipage}
        \end{center}
    \end{figure}
\end{exa}
\FloatBarrier
We impose the following kernel conditions on the coefficients to account for any redundancy in the presentation of the self-adjoint hereditary quadratic.

\begin{defn}
    We say that a line system \((L,H,K)\) is \emph{minimal} for \(f\), or \(f\) is \emph{minimally presented} by the line system \((L,H,K)\), if \(f = L + H^*H - K^*K\) and 
    \begin{itemize}
        \item[(i)] \(\bigcap_{i\in[g]} \ker L_i \cap \ker H_i \cap \ker K_i = \{0\}\), and 
        \item[(ii)] \(\bigcap_{i\in[g]} \ker\begin{pmatrix} H_i^* & K_i^*\end{pmatrix} = \{0\}\).
    \end{itemize}
    We may also say that \((L,H,K)\) is minimal, with \(f\) implicit.
\end{defn}

Every line system \((L,H,K)\) with associated quadratic \(f=L+H^*H-K^*K\) can be replaced by a minimal line system:
\begin{itemize}
    \item[(i)] If \(0\neq v\in\bigcap_{i\in[g]} \ker L_i \cap \ker H_i \cap \ker K_i\), then let \(P_v\) be the orthogonal projection away from \(\spn\{v\}\) and consider the compression \(P_vfP_v\), with line system \((P_vLP_v,P_vHP_v,P_vKP_v)\).
    \item[(ii)] If \(0\neq\begin{pmatrix}v\\-w \end{pmatrix} \in \ker \begin{pmatrix} H_i^* & K_i^*\end{pmatrix}\) then \(H_i^*v=K_i^*w\) for all \(i\in[g]\) and either \(v\neq0\) or \(w\neq0\). If \(v\neq0\), consider the same quadratic \(f\) with line system \((L,P_vH,K)\) where \(P_v\) is the orthogonal projection away from \(\spn\{v\}\). If \(w\neq0\), consider the same quadratic \(f\) with line system \((L,H,P_wK)\) where \(P_w\) is the orthogonal projection away from \(\spn\{w\}\).
\end{itemize}
We make the following observation to mitigate any concern for the positivity domain when compressing \(f\).

Let \(S=\bigcap_{i\in[g]} \ker L_i \cap \ker H_i \cap \ker K_i\).
Let \(P\) be the orthogonal projection onto \(S^\perp\) and \(Q\) the orthogonal projection onto \(S\).
If \(S\) is nonzero and \(\restr{L_0}{S}=QL_0Q\) is not positive semidefinite, then \(\mathcal{D}(f)=\varnothing\).
So whenever \(\mathcal{D}(f)\neq\varnothing\) and \(S\) is nonzero, \(\restr{L_0}{S}\) must be positive semidefinite. 
Set \(L'=L_0+\sum_{i\in[g]}L_ix_i\).
Using the decomposition \(\mathbb{C}^d = S \oplus S^\perp\), we see that 
\[
    f(X,Y) = \begin{pmatrix}
        (Q\otimes I_n)(L_0\otimes I_n)(Q\otimes I_n) & (Q\otimes I_n)L'(X)(P\otimes I_n) \\
        (P\otimes I_n)(L'(X))^*(Q\otimes I_n) & (P\otimes I_n)f(X,Y)(P\otimes I_n)
    \end{pmatrix}\quad\forall X,Y\in\left(\M_n(\mathbb{C})\right)^g.
\]
Since \((Q\otimes I_n)(L_0\otimes I_n)(Q\otimes I_n)=\restr{L_0}{S} \otimes I_n\) is positive semidefinite, we can taking a Schur complement using the notation \(A^{(-1)}\) for pseudoinverse: 
\begin{align*}
    \text{SC}:\!&=(P\otimes I_n)f(X,Y)(P\otimes I_n)-(P\otimes I_n)(L'(X))^*(Q\otimes I_n)\left(\restr{L_0}{S}\otimes I_n\right)^{(-1)}(Q\otimes I_n)L'(X)(P\otimes I_n) \\
    &=(P\otimes I_n)\tilde{f}(X,Y)(P\otimes I_n)
\end{align*}
where \(\tilde{f}(X,Y)=f(X,Y)-(L'(X))^*(Q\otimes I_n)\left(\restr{L_0}{S}\otimes I_n\right)^{(-1)}(Q\otimes I_n)L'(X)\).
Thus 
\[\mathcal{D}(f)=\mathcal{D}(P\tilde{f}P)\]
where \(P\tilde{f}P\) is a self-adjoint hereditary quadratic polynomial with coefficients in \(\text{End}(S^\perp)\).
Hence we have shown the following.
\begin{prop}
    If \(\mathcal{S}\) is a nonempty positivity domain of a self-adjoint hereditary quadratic polynomial, then there exists a minimal line system \((L,H,K)\) with associated polynomial \(f\) for which \(\mathcal{S} = \mathcal{D}(f)\).
\end{prop}

Minimality is preserved by constant shifts: if \((L,H,K)\) is a minimal line system then so is \((C+L,H,K)\) for any constant matrix \(C\in\M_d(\mathbb{C})\) (of course, if we fix \(f\) then we must be careful.)
Minimality is also preserved under \(^*\)-congruent action of \(\GL_d(\mathbb{C})\).

\begin{prop}\label{GLactLS}
    Let \((L,H,K)\) a minimal line system. 
    For any \(S\in\GL_d(\mathbb{C})\), the line system \((S^*LS,HS,KS)\) is also minimal.
\end{prop}

\begin{proof}
    We compute the kernels of the coefficients in the line system \((S^*LS,HS,KS)\):
    \begin{align*}
        \ker S^*L_iS = S^{-1}\left(\ker L_i\right), \quad
        \ker H_iS = S^{-1}\left(\ker H_i\right), \quad
        \ker K_iS = S^{-1}\left(\ker K_i\right).
    \end{align*}
    Consequently, 
    \[
        \bigcap_{i}\ker S^*L_iS\cap\ker H_iS\cap\ker K_iS = S^{-1}\left(\bigcap_{i}\ker L_i\cap\ker H_i\cap\ker K_i\right)=S^{-1}(0)=0.
    \]
    Also \(\ker\begin{pmatrix} S^*H_i^* & S^*K_i^*\end{pmatrix} = S^*\left(\ker\begin{pmatrix} H_i^* & K_i^*\end{pmatrix}\right)\) so, since \(S^*\) is invertible,
    \[
        \bigcap_{i}\ker\begin{pmatrix}S^*H_i^* & S^*K_i^*\end{pmatrix} = \bigcap_iS^*\left(\ker\begin{pmatrix} H_i^* & K_i^*\end{pmatrix}\right) = S^*\left(\ker\begin{pmatrix} H_i^* * K_i^*\end{pmatrix}\right) = S^*\left(0\right) = 0.
    \]
\end{proof}

We introduce a linearisation of \(f=L+H^*H-K^*K\):
\[
    \mathcal{L}_f\defeq\begin{pmatrix}
        I & K & 0 \\
        K^* & L & H^* \\
        0 & H & -I
    \end{pmatrix}.
\]
This linearisation \(\mathcal{L}_f\) is a Hermitian pencil, and \(\mathcal{L}_f(0)\) is invertible whenever \(f(0)\) is invertible. 
Notice that \(\mathcal{L}_f\) is never positive semidefinite. 
Nevertheless \(\mathscr{Z}(\mathcal{L}_f)=\mathscr{Z}(f)\), since \(f\) and \(\mathcal{L}_f\) are stably associated:
\[
    \begin{pmatrix}
        I & 0 \\
        0 & L+H^*H-K^*K
    \end{pmatrix} = 
    \begin{pmatrix}
        I & 0 \\ 
        -K^* & I 
    \end{pmatrix}
    \begin{pmatrix}
        I & K \\
        K^* & L+H^*H
    \end{pmatrix}
    \begin{pmatrix}
        I & -K \\
        0 & I
    \end{pmatrix},
\] 
and 
\[
    \begin{pmatrix}
        I & K & 0 \\
        K^* & L+H^*H & 0 \\
        0 & 0 & I
    \end{pmatrix} = 
    \begin{pmatrix}
        I & 0 & 0 \\
        0 & I & H^* \\
        0 & 0 & I
    \end{pmatrix}
    \begin{pmatrix}
        I & K & 0 \\
        K^* & L & H^* \\
        0 & H & -I 
    \end{pmatrix}
    \begin{pmatrix}
        I & 0 & 0 \\
        0 & I & 0 \\
        0 & H & -I
    \end{pmatrix}.
\]

\begin{lem}\label{atomirred}
    Suppose \(f\) is an atomic self-adjoint hereditary quadratric with minimal line system \((L,H,K)\). 
    Then the linearisation \(\mathcal{L}_f\) is irreducible.
\end{lem}
\begin{proof}
    For each \(i\in\{1,\ldots,g\}\), let \(A_i\) be the coefficient of \(x_i\) in \(\mathcal{L}_f\), i.e.,
    \[
        A_i = \begin{pmatrix}
            0 & K_i & 0 \\
            0 & L_i & 0 \\
            0 & H_i & 0
        \end{pmatrix}\quad\text{and}\quad A_i^* = \begin{pmatrix}
            0 & 0 & 0 \\
            K_i^* & L_i^* & H_i^* \\
            0 & 0 & 0
        \end{pmatrix}.
    \]
    If \(\begin{pmatrix}u\\v\\w\end{pmatrix}\) is in the common kernel of the \(A_i\) and \(A_i^*\) for all \(i\), then 
    \[
        \begin{pmatrix}
            K_iv \\ L_iv \\ H_i v
        \end{pmatrix} = 0\quad\text{and}\quad
        \begin{pmatrix}
            0 \\ K_i^*u + L_i^*v + H_i^*w \\ 0
        \end{pmatrix} = 0.
    \]
    Then \(v\in\bigcap_{i\in\{1,\ldots,g\}}\left(\ker L_i\cap\ker H_i\cap\ker K_i\right)\).
    But \((L,H,K)\) is minimal, so \(v=0\). 
    Consequently \(K_i^*u + H_i^* w=0\) so \(\begin{pmatrix} w\\u\end{pmatrix}\in\bigcap_{i\in\{1,\ldots,g\}}\ker\begin{pmatrix} H_i^* & K_i^*\end{pmatrix}\); again by minimality, \(u=w=0\). 
    So the coefficients \(A_i\) and \(A_i^*\) have no nontrivial common kernel.
    By \cite[Theorem 5.8.8]{Cohn06}, \(\mathcal{L}_f\) is irreducible if and only if \(\mathcal{L}_f\) is an atom.
    Since \(f\) is an atom and \(\mathcal{L}_f\) is stably associated to \(f\), then \(\mathcal{L}_f\) is an atom.
    So \(\mathcal{L}_f\) is irreducible.
\end{proof}

\begin{thm}\label{penciltopoly}
    Let \((L,H,K)\) and \((\tilde L,\tilde H, \tilde K)\) be two minimal line systems with associated polynomials \(f=L+H^*H-K^*K\) and \(\tilde f=\tilde L+\tilde H^*\tilde H-\tilde K^*\tilde K\), and linearisations \[\mathcal{L}:=\mathcal{L}_f=\begin{pmatrix}I&K&0\\K^*&L&H^*\\0&H&-I\end{pmatrix},\quad\mathcal{\tilde L}:=\mathcal{L}_{\tilde f}=\begin{pmatrix}I&\tilde K&0\\\tilde K^*&\tilde L&\tilde H^*\\0&\tilde H&-I\end{pmatrix}.\]
    If \(S\mathcal{L} S^*=r\mathcal{\tilde{L}}\) for some invertible matrix \(S\) and \(r\in\mathbb{R}^\times\), then \(\hat Sf\hat S^*=r\tilde{f}\) for some \(\hat S\in \GL_d(\mathbb{C})\).
\end{thm}

\begin{proof}
    Write \(S\) as a \(3\times 3\) block matrix: \[S=\begin{pmatrix}S_{11}&S_{12}&S_{13}\\S_{21}&S_{22}&S_{23}\\S_{31}&S_{32}&S_{33}\end{pmatrix}.\]
    Then \(S\mathcal{L}S^*=r\mathcal{\tilde L}\) gives us the following equations:
    \begin{align*}
        S_{11}S_{11}^*+S_{12}K^*S_{11}^*+S_{11}KS_{12}^*+S_{12}LS_{12}^*+S_{13}HS_{12}^*+S_{12}H^*S_{13}^*-S_{13}S_{13}^*&=rI, \tag{1,1}\\
        S_{11}S_{21}^*+S_{12}K^*S_{21}^*+S_{11}KS_{22}^*+S_{12}LS_{22}^*+S_{13}HS_{22}^*+S_{12}H^*S_{23}^*-S_{13}S_{23}^* &= r\tilde K, \tag{1,2}\\
        S_{11}S_{31}^*+S_{12}K^*S_{31}^*+S_{11}KS_{32}^*+S_{12}LS_{32}^*+S_{13}HS_{32}^*+S_{12}H^*S_{33}^*-S_{13}S_{33}^* &= 0, \tag{1,3}\\
        S_{21}S_{21}^*+S_{22}K^*S_{21}^*+S_{21}KS_{22}^*+S_{22}LS_{22}^*+S_{23}HS_{22}^*+S_{22}H^*S_{23}^*-S_{23}S_{23}^* &= r\tilde L, \tag{2,2}\\
        S_{21}S_{31}^*+S_{22}K^*S_{31}^*+S_{21}KS_{32}^*+S_{22}LS_{32}^*+S_{23}HS_{32}^*+S_{22}H^*S_{33}^*-S_{23}S_{33}^* &= r\tilde H^*, \tag{2,3}\\
        S_{31}S_{31}^*+S_{32}K^*S_{31}^*+S_{31}KS_{32}^*+S_{32}LS_{32}^*+S_{33}HS_{32}^*+S_{32}H^*S_{33}^*-S_{33}S_{33}^* &= -rI. \tag{3,3}
    \end{align*}
    By looking at the coefficient of \(x_i\) in equation (1,1), and the coefficients of \(x_i^*\) in equations (1,2) and (1,3), we have 
    \begin{align*}
        S_{11}K_iS_{12}^*+S_{12}L_iS_{12}^*+S_{13}H_iS_{12}^* &= 0,\\
        S_{12}K_i^*S_{21}^*+S_{12}L_i^*S_{22}^*+S_{12}H_i^*S_{23}^* &= 0,\\
        S_{12}K_i^*S_{31}^*+S_{12}L_i^*S_{32}^*+S_{12}H_i^*S_{33}^* &= 0.
    \end{align*}
    Taking the star of the second and third equation in the above block, and writing in matrix form,  
    \[
        S\begin{pmatrix}K_i\\L_i\\H_i\end{pmatrix}S_{12}^* = 0.
    \]
    Since \(S\) is invertible,
    \[
        \begin{pmatrix}K_i\\L_i\\H_i\end{pmatrix}S_{12}^* = 0.
    \]
    By minimality of \((L,H,K)\), the \(K_i,L_i,H_i\) cannot have any common kernels. 
    Thus \(S_{12}=0\).
    Similarly by looking at the coefficients of \(x_i\) in equations (1,3) and (2,3), and the coefficient of \(x_i^*\) in equation (3,3), we get
    \begin{align*}
        S_{11}K_iS_{32}^*+S_{12}L_iS_{32}^*+S_{13}H_iS_{32}^* &= 0,\\
        S_{21}K_iS_{32}^*+S_{22}L_iS_{32}^*+S_{23}H_iS_{32}^* &= 0,\\
        S_{32}K_i^*S_{31}^*+S_{32}L_i^*S_{32}^*+S_{32}H_i^*S_{33}^* &= 0.
    \end{align*}
    Taking the star of the third equation in the above block, and writing in matrix form,
    \[
        S\begin{pmatrix}K_i\\L_i\\H_i\end{pmatrix}S_{32}^* = 0.
    \]
    Again since \(S\) is invertible and by minimality of \((L,H,K)\), we get that \(S_{32}=0\).
    Now equating the coefficients of \(x_i\) in equation (1,2), \[S_{11}K_iS_{22}^*+S_{13}H_iS_{22}^*=r\tilde K_i.\]
    Taking the star and writing in block matrix form, \[S_{22}\begin{pmatrix}K_i^*&H_i^*\end{pmatrix}\begin{pmatrix}S_{11}^*\\S_{13}^*\end{pmatrix}=r\tilde K_i^*.\]
    Similarly from equation (2,3), the coefficients of \(x_i^*\) tell us that \[S_{22}\begin{pmatrix}K_i^*&H_i^*\end{pmatrix}\begin{pmatrix}S_{31}^*\\S_{33}^*\end{pmatrix}=r\tilde H_i^*.\]
    It follows that \[S_{22}\begin{pmatrix}K_i^*&H_i^*\end{pmatrix}\begin{pmatrix}S_{11}^*&S_{31}^*\\S_{13}^*&S_{33}^*\end{pmatrix}=r\begin{pmatrix}\tilde K_i^*&\tilde H_i^*\end{pmatrix}.\]
    By minimality of \((L,H,K)\), the kernels of \(\begin{pmatrix}\tilde K_i^*&\tilde H_i^*\end{pmatrix}\) for each \(i\) must intersect to zero. 
    Consequently, the above equation implies that \(\begin{pmatrix}S_{11}^*&S_{31}^*\\S_{13}^*&S_{33}^*\end{pmatrix}\) is invertible.
    But equating the constant terms in equations (1,2) and (2,3), we get
    \begin{align*}
        S_{21}S_{11}^*-S_{23}S_{13}^* &= 0,\\
        S_{21}S_{31}^*-S_{23}S_{33}^* &= 0.
    \end{align*}
    That is, \[\begin{pmatrix}S_{21}&-S_{23}\end{pmatrix}\begin{pmatrix}S_{11}^*&S_{31}^*\\S_{13}^*&S_{33}^*\end{pmatrix}=0.\]
    Thus \(S_{21}=0\) and \(S_{23}=0\).
    So our matrix \(S\) has the form \[S=\begin{pmatrix}S_{11}&0&S_{13}\\0&S_{22}&0\\S_{31}&0&S_{33}\end{pmatrix}\] where \(S_{22}\) must be invertible since \(S\) is invertible.
    The equations of \(S\mathcal{L}S^*=r\mathcal{\tilde L}\) are now 
    \begin{align*}
        S_{11}S_{11}^*-S_{13}S_{13}^* &= rI,\tag{1}\\
        S_{11}KS_{22}^*+S_{13}HS_{22}^* &= r\tilde K,\tag{2}\\
        S_{11}S_{31}^*-S_{13}S_{33}^* &= 0, \tag{3}\\
        S_{22}LS_{22}^* &= r\tilde L, \tag{4}\\
        S_{22}K^*S_{31}^*+S_{22}H^*S_{33}^* &= r\tilde H^*, \tag{5}\\
        S_{31}S_{31}^*-S_{33}S_{33}^* &= -rI. \tag{6}
    \end{align*}
    From equations (1), (3) and (6), 
    \[\begin{pmatrix}S_{11}&S_{13}\\S_{31}&S_{33}\end{pmatrix}\begin{pmatrix}I&0\\0&-I\end{pmatrix}\begin{pmatrix}S_{11}^*&S_{31}^*\\S_{13}^*&S_{33}^*\end{pmatrix}=r\begin{pmatrix}I&0\\0&-I\end{pmatrix}\] and since \(\begin{pmatrix}I&0\\0&-I\end{pmatrix}\) is an involution then \[\begin{pmatrix}I&0\\0&-I\end{pmatrix}\begin{pmatrix}S_{11}&S_{13}\\S_{31}&S_{33}\end{pmatrix}\begin{pmatrix}I&0\\0&-I\end{pmatrix}\begin{pmatrix}S_{11}^*&S_{31}^*\\S_{13}^*&S_{33}^*\end{pmatrix}=rI.\]
    Consequently 
    \[
        \displaystyle\frac{1}{\sqrt{r}}\begin{pmatrix}I&0\\0&-I\end{pmatrix}\begin{pmatrix}S_{11}&S_{13}\\S_{31}&S_{33}\end{pmatrix} \quad\text{ and }\quad \displaystyle\frac{1}{\sqrt{r}}\begin{pmatrix}I&0\\0&-I\end{pmatrix}\begin{pmatrix}S_{11}^*&S_{31}^*\\S_{13}^*&S_{33}^*\end{pmatrix}
    \]
    are inverses of one another, so \[\begin{pmatrix}I&0\\0&-I\end{pmatrix}\begin{pmatrix}S_{11}^*&S_{31}^*\\S_{13}^*&S_{33}^*\end{pmatrix}\begin{pmatrix}I&0\\0&-I\end{pmatrix}\begin{pmatrix}S_{11}&S_{13}\\S_{31}&S_{33}\end{pmatrix}=r\begin{pmatrix}I&0\\0&I\end{pmatrix}.\]
    That is, \[\begin{pmatrix}S_{11}^*S_{11}-S_{31}^*S_{31} & S_{11}^*S_{13}-S_{31}^*S_{33} \\ S_{33}^*S_{31}-S_{13}^*S_{11} & S_{33}^*S_{33}-S_{13}^*S_{13}\end{pmatrix}=r\begin{pmatrix}I&0\\0&I\end{pmatrix}\]
    and so \(S_{11}^*S_{11}-S_{31}^*S_{31}=rI\), \(S_{11}^*S_{13}-S_{31}^*S_{33}=0\), \(S_{33}^*S_{31}-S_{13}^*S_{11}=0\) and \(S_{33}^*S_{33}-S_{13}^*S_{13}=rI\).
    Using equations (2) and (5), 
    \begin{align*}
        r^2S_{22}^{-1}\left(\tilde H^*\tilde H-\tilde K^*\tilde K\right)S_{22}^{-*} = K^*&(S_{31}^*S_{31}-S_{11}^*S_{11})K+K^*(S_{31}^*S_{33}-S_{11}^*S_{13})H \\ &+H^*(S_{33}^*S_{31}-S_{13}^*S_{11})K + H^*(S_{33}^*S_{33}-S_{13}^*S_{13})H.
    \end{align*} 
    It follows that
    \begin{align*}
        r^2S_{22}^{-1}\left(\tilde H^*\tilde H-\tilde K^*\tilde K\right)S_{22}^{-*} &= K^*(-rI)K+K^*(0)H+H^*(0)K+H^*(rI)H \\
        &= r(H^*H-K^*K).
    \end{align*}
    So \(S_{22}(H^*H-K^*K)S_{22}^*=r(\tilde{H}^*\tilde{H}-\tilde{K}^*\tilde{K})\). 
    From equation (4), \(S_{22}LS_{22}^*=r\tilde{L}\) so \[S_{22}fS_{22}^*=r\tilde{f}.\]
\end{proof}

\begin{thm}\label{positivstellensatz}
    Let \(f\) and \(\tilde f\) be unsignatured atomic hereditary quadratic polynomials, with \(f(0)\) and \(\tilde f(0)\) positive definite.
    Then \[\mathcal{D}(f)=\mathcal{D}(\tilde f)\iff f=S\tilde{f}S^*\text{ for some }S\in \GL_d(\mathbb{C}).\]
\end{thm}

\begin{proof}
    If \(f=S\tilde f S^*\) for some \(S\in \GL_d(\mathbb{C})\) then clearly \(\mathcal{D}(f)=\mathcal{D}(\tilde f)\).
    
    For the converse, suppose \(\mathcal{D}(f)=\mathcal{D}(\tilde f)\) and let \((L,H,K)\), \((\tilde{L},\tilde{H},\tilde{K})\) be minimal line systems for \(f\), \(\tilde f\) respectively.
    Since \(\mathcal{D}(f)=\mathcal{D}(\tilde f)\), then \(\partial\mathcal{D}(f)=\partial\mathcal{D}(\tilde f)\)
    and so by Theorem \ref{posdense}, there exists \(n_0\in\mathbb{N}\) for which \(\mathscr{Z}_n(f)=\mathscr{Z}_n(\tilde f)\) for all \(n\in n_0\cdot\mathbb{N}\).
    Since \(f(0)>0\) and \(\tilde f(0)>0\), we can choose \(n_0\) so that \(\mathscr{Z}_n(f)=\mathscr{Z}_n(\tilde f)\) for all \(n\geq n_0\).
    Recall the linearisations for \(f\) and \(\tilde f\) given by
    \[
        \mathcal{L}\defeq\mathcal{L}_f=\begin{pmatrix}I&K&0\\K^*&L&H^*\\0&H&-I\end{pmatrix},\quad
        \mathcal{\tilde L}\defeq\mathcal{L}_{\tilde f}=\begin{pmatrix}I&\tilde K&0\\\tilde K^*&\tilde L&\tilde H^*\\0&\tilde H&-I\end{pmatrix}.
    \]
    Since \(\mathscr{Z}(f)=\mathscr{Z}(\mathcal{L})\) and \(\mathscr{Z}(\tilde f)=\mathscr{Z}(\mathcal{\tilde L})\) then \(\mathscr{Z}_n(\mathcal{L})=\mathscr{Z}_n(\mathcal{\tilde L})\) for all \(n\geq n_0\). 
    If there were some \(k\in\mathbb{N}\) with \(k<n_0\) for which \(\mathscr{Z}_k(\mathcal{L})\neq\mathscr{Z}_k(\mathcal{\tilde L})\), then either \(\mathscr{Z}_k(\mathcal{L})\backslash\mathscr{Z}_k(\mathcal{\tilde L})\neq\varnothing\) or \(\mathscr{Z}_k(\mathcal{\tilde L})\backslash\mathscr{Z}_k(\mathcal{L})\neq\varnothing\). 
    Suppose \(\mathscr{Z}_k(\mathcal{L})\backslash\mathscr{Z}_k(\mathcal{\tilde L})\neq\varnothing\) (the other case follows similarly), then there is some \(Y\in \M_k(\mathbb{C})^g\) for which \(\det\mathcal{L}(Y)=0\) but \(\det\mathcal{\tilde L}(Y)\neq0\).
    Considering the Kronecker product \(I_{n_0}\otimes Y\in \M_{kn_0}(\mathbb{C})^g\), we get that 
    \[
        \det\mathcal{L}(I_{n_0}\otimes Y)=\det\bigoplus_{i=1}^{n_0}\mathcal{L}(Y)=\left(\det\mathcal{L}(Y)\right)^{n_0}=0,
    \] 
    while \(\det\mathcal{\tilde L}(I_{n_0}\otimes Y)\neq0\) by a similar computation.
    This is a contradiction: \(n_0k> n_0\), and \(\mathscr{Z}_n(\mathcal{L})=\mathscr{Z}_{n}(\mathcal{\tilde L})\) for all \(n\geq n_0\).
    Thus we must have that \(\mathscr{Z}(\mathcal{L})=\mathscr{Z}(\mathcal{\tilde L})\). 
    By Lemma \ref{atomirred}, \(\mathcal{L}\) and \(\mathcal{\tilde L}\) are irreducible, 
    so by \cite[Theorem 2.11]{HKV20}, there exists \(P,Q\in\text{GL}_m(\mathbb{C})\) such that \(\mathcal{\tilde L}=P\mathcal{L}Q\).
    However, by the self-adjoint property of \(\mathcal{\tilde L}\) and \(\mathcal{L}\), 
    \[
        \mathcal{\tilde L}^*=Q^*\mathcal{L}P^*=Q^*P^{-1}\mathcal{\tilde L}Q^{-1}P^*=(Q^*P^{-1})\mathcal{\tilde L}(P^{-*}Q)^{-1}.
    \]
    By \cite[Lemma 2.10]{HKV20}, we must have \(Q^*P^{-1}=t I\) and \(P^{-*}Q=t I\) for some \(t\in\mathbb{C}^\times\).
    That is, \(Q^*=t P\) and \(Q=t P^*\); it follows that \(t\) is real.
    Consequently, \(\mathcal{\tilde L}=t P\mathcal{L}P^*\) where \(t\in\mathbb{R}^\times\).
    Writing \(P_0=\sqrt{\abs{t}}P\), we get \(\mathcal{\tilde L}=\pm P_0\mathcal{L}P_0^*\) where \(P_0\in\text{GL}_m(\mathbb{C})\).
    By Theorem \ref{penciltopoly}, \(f=\pm S\tilde f S^*\) for some \(S\in\text{GL}_n(\mathbb{C})\).
    If \(f=-S\tilde fS^*\), then \(\mathcal{D}(\tilde f)=\mathcal{D}(f)=\mathcal{D}(-\tilde f)\).
    But \(\tilde f(0)>0\), so \(f=S\tilde{f}S^*\).
\end{proof}

\begin{comment}
The above theorem only applies to atoms, but does not work on e.g. Example \ref{annulus}. 
However, the same result holds for self-adjoint hereditary quadratics of the form \(f_1\oplus\cdots\oplus f_r\) where \(f_1,\ldots,f_r\) are atoms.
\begin{cor}
    Let \(f=f_1\oplus\cdots\oplus f_r\) and \(g=g_1\oplus\cdots\oplus g_s\) be positivity-minimal and unsignatured with \(f(0)\) and \(g(0)\) positive definite.
    Suppose that \(f_1,\ldots,f_r\) and \(g_1,\ldots,g_s\) are atomic self-adjoint hereditary quadratics. 
    Then 
    \[
        \mathcal{D}(f)=\mathcal{D}(g)\iff f=SgS^*\text{ for some }S\in\GL_d(\mathbb{C}).
    \]
\end{cor}
\begin{proof}
    
\end{proof}
\end{comment}

%%%%%%%%%%%%%%%%%%%%%%%%%%%%%%%%%%%%%%%%%%%%%%%%%%%%
% GAMMA CONVEXITY
%%%%%%%%%%%%%%%%%%%%%%%%%%%%%%%%%%%%%%%%%%%%%%%%%%%%
\section{Gamma Convexity}\label{GammaConvexity}
Hereditary quadratic noncommutative polynomials appear in the theory of \(\Gamma\)-convexity. We will begin by outlining some of the theory that can be found in \cite{JKMMP20}. 
The following are definitions common in the theory of noncommutative functions.
We will work over the field \(\mathbb{C}\). 
Denote the disjoint union of matrix tuples by
\[
    \mathcal{M}_g:=\bigsqcup_{n\in\mathbb{N}}\M_n(\mathbb{C})^g.
\]

\begin{defn}
    A \emph{free set} is a subset \(\Omega\subset\mathcal{M}_g\) with the following properties:
    \begin{itemize}
        \item[(i)] closed under direct sum: if \(X=(X_1,\ldots,X_g)\in\Omega\cap\M_n(\mathbb{C})^g\) and \(Y=(Y_1,\ldots,Y_g)\in\Omega\cap\M_m(\mathbb{C})^g\) then \(X\oplus Y=(X_1\oplus Y_1,\ldots,X_g\oplus Y_g)\in\Omega\cap\M_{n+m}(\mathbb{C})^g\);
        \item[(ii)] closed under simultaneous unitary similarity: if \(X=(X_1,\ldots,X_g)\in\Omega\cap\M_n(\mathbb{C})^g\) and \(U\in \M_n(\mathbb{C})\) unitary, then \(U^*XU=(U^*X_1U,\ldots,U^*X_gU)\in\Omega\cap\M_n(\mathbb{C})^g\);
        \item[(iii)] closed under reducing subspaces: if \(X=(X_1,\ldots,X_g)\in\Omega\cap\M_n(\mathbb{C})^g\) and \(S\subset\mathbb{C}^n\) a common reducing subspace for \(X_1,\ldots,X_g\), then \(\restr{X}{S}\,=(\restr{X_1}{S}\,,\ldots,\restr{X_g}{S}\,)\in\Omega\).
    \end{itemize}
\end{defn}

Let \(\Gamma\) be a finite tuple of noncommutative polynomials containing every degree one monomial, and closed under involution. 
That is, for some \(r\in\mathbb{N}\) and \(\gamma_1,\ldots,\gamma_r\in\mathbb{C}\mon\),
\[
    \Gamma = \left(x_1,\ldots,x_g,y_1,\ldots,y_g,\gamma_1,\gamma_1^*,\ldots,\gamma_r,\gamma_r^*\right).\tag{\(\ddagger\)}
\]

\begin{defn}
    A pair \((X,V)\), where \(X\in\M_n(\mathbb{C})^g\) and \(V:\mathbb{C}^m\to\mathbb{C}^n\) an isometry for some \(m,n\in\mathbb{N}\), is a \(\Gamma\)-pair if 
    \[
        V^*\gamma(X,X^*)V=\gamma(V^*XV,V^*X^*V)\quad\forall \gamma\in\Gamma.
    \]
\end{defn}
Let \(\mathcal{C}_\Gamma\) be the collection of all \(\Gamma\)-pairs.
\begin{defn}
    A subset \(\mathcal{K}\subset\mathcal{M}_g\) is \emph{\(\Gamma\)-convex} if \(\mathcal{K}\) is free and 
    \[
       (X,V)\in\mathcal{C}_\Gamma\text{ and } X\in\mathcal{K}\implies V^*XV\in\mathcal{K}.
    \]
\end{defn}

\begin{defn}
    A \emph{\(\Gamma\)-pencil}, for \(\Gamma\) of the form \((\ddagger)\), is a self-adjoint matrix polynomial \(L\in\M_d(\mathbb{C}\mon)\) of the form
    \[
        L = A_0 + \sum_{i=1}^{g}(A_ix_i+A_i^*y_i)+\sum_{j=1}^{r}(B_j\gamma_j+B_j^*\gamma_j^*).
    \]
    Here, \(A_0\) is Hermitian.
\end{defn}

\begin{exa}\label{hqp is gpen}
    Let \(\gamma_{(j-1)g+k}=y_jx_k\) for all \(j,k\in\{1,\ldots,g\}\).
    For this choice of \(\Gamma\), a \(\Gamma\)-pencil is a self-adjoint matrix polynomial 
    \[
        f = A_0 + \sum_{j=1}^{g}A_jx_j + \sum_{j=1}^{g}A_j^*y_j + \sum_{j,k\in\left[g\right]}C_{j,k}y_jx_k
    \]
    where \(C_{j,k}^*=C_{k,j}\).
    These are precisely our hereditary quadratic polynomials.
    This particular choice of \(\Gamma\) will be called \(\Gamma_{yx}\) for brevity.
\end{exa}

\begin{rem}\label{gpengconv}
    For a \(\Gamma\)-pencil \(L\), it is clear that \(\mathcal{D}(L)\) is (or, can be identified with) a \(\Gamma\)-convex set.
    The positivity domain is a free set, so all one needs to check is if \((V^*XV,V^*X^*V)\in\mathcal{D}(L)\) for each \(\Gamma\)-pair \((X,V)\) with \(X\in\mathcal{D}(L)\), but this follows from \(\gamma(V^*XV,V^*X^*V)=V^*\gamma(X,X^*)V\) so that
    \begin{align*}
        L(V^*XV,V^*X^*V) &= A_0\otimes I_m + \sum_{\gamma\in\Gamma} A_\gamma\otimes\gamma(V^*XV,V^*X^*V) \\
        &= A_0\otimes V^*V + \sum_{\gamma\in\Gamma}A_\gamma\otimes V^*\gamma(X,X^*)V \\
        &= (I_d\otimes V^*)L(X,X^*)(I_d\otimes V).
    \end{align*}
\end{rem}

We now classify \(\Gamma\)-pairs for the choice of \(\Gamma\) in Example \ref{hqp is gpen}.
\begin{prop}
    Let \(f\in\M_d(\mathbb{C}\mon)\) be a self-adjoint hereditary quadratic.
    Let \((\underline{X},W)\) be a \(\Gamma_{xy}\)-pair.
    Then either \(W\) is a unitary matrix, or \(\underline{X}=(X_1,\ldots,X_g)\in\M_n(\mathbb{C})^g\) have a nontrivial common invariant subspace and \(W\) projects onto the invariant subspace.
    % That is, if \((\underline{X},W)\in\mathcal{C}_\Gamma\) then 
    % \[
    %     (\underline{X},\underline{X}^*)\in\mathcal{D}(f) \implies \underline{X}=(X_1,\ldots,X_g) \text{ have a common invariant subspace.}
    %\]
\end{prop}

\begin{proof}
    Let \((\underline{X},W)\in\mathcal{C}_{\Gamma_{xy}}\) so that 
    \[
        W^*X_jX_iW = W^*X_j^*WW^*X_iW\quad\forall i,j\in\{1,\ldots,g\}.\tag{\(\dagger\)}
    \]
    Since \(W^*W=I\) then \(W=U^*\begin{pmatrix}I_m\\0\end{pmatrix}V\) for some unitaries \(U,V\) (by SVD). 
    Decomposing \(UX_iU^*\) along \(\mathbb{C}^n\cong\mathbb{C}^m\oplus\mathbb{C}^{n-m}\),
    \[
        U^*X_iU = \begin{pmatrix} 
            A_i & B_i \\
            C_i & D_i 
        \end{pmatrix}.
    \]  
    Then (\(\dagger\)) becomes \(C_j^*C_i = 0\) for all \(i,j\in\{1,\ldots,g\}\).
    In particular, \(i=j\) gives \(C_i=0\) for all \(i\in\{1,\ldots,g\}\) so that \(UX_iU^*\) are all block-upper-triangular.
    Consequently, \(UX_iU^*\) have common invariant subspace isomorphic to \(\mathbb{C}^{n-m}\).
\end{proof}

The theory of \(\Gamma\)-convexity has a hyperplane separation theorem.
The following is \cite[Theorem 2.4]{JKMMP20}.

\begin{thm}[Jury-Klep-Mancuso-McCullough-Pascoe]
    Let \(\mathcal{K}\) be a \(\Gamma\)-convex set containing \(0\), and suppose \(\Gamma(0)=0\).
    Let \(\mathcal{C}\) be the matrix-convex hull of \(\Gamma(\mathcal{K})\).
    If there exists \(Y\in\M_n(\mathbb{C})^g\) so that \(\Gamma(Y)\notin\overline{\mathcal{C}}\), then there exists a \(\Gamma\)-pencil \(L\) of size \(n\) such that \(L(\mathcal{K})\) is positive definite, but \(L(Y)\) is not positive semidefinite. 
    In particular, if \(\mathcal{C}\) is closed, then for each \(Y\notin\mathcal{K}(n)\) there exists a monic \(\Gamma\)-pencil \(L\) of size \(n\) such that \(L(\mathcal{K})\) is positive definite, but \(L(Y)\) is not positive semidefinite.
\end{thm}

In particular, this implies that \(\Gamma_{yx}\)-convex sets are intersections of self-adjoint hereditary quadratics.

{\bf Acknowledgments.} The authors would like to thank James E. Pascoe for his assistance with \(\Gamma\)-convexity.

\printbibliography[heading=bibintoc,title={References}]

\end{document}